\documentclass[11pt,reqno]{amsart}

\usepackage[english]{babel}
\usepackage[utf8]{inputenc}
\usepackage[T1]{fontenc}
\usepackage{lmodern} \normalfont 
\DeclareFontShape{T1}{lmr}{bx}{sc} { <-> ssub * cmr/bx/sc }{}
\usepackage{microtype}	

\usepackage{amssymb}
\usepackage{amsmath}
\usepackage{amsthm}
\usepackage{mathtools}
\usepackage{mathrsfs}
\usepackage{accents} 
\usepackage{etoolbox}
\usepackage{siunitx}
\usepackage{dsfont} 
\usepackage{graphicx}
\usepackage{color}
\usepackage[dvipsnames]{xcolor}
\usepackage{tikz}
\usetikzlibrary{calc,positioning,shapes}
\usetikzlibrary{patterns,decorations.pathmorphing,decorations.markings}
\usepackage{pgfplots}
\pgfplotsset{compat=newest}
\usepackage[margin=10pt,font=small,labelfont=bf,labelsep=endash]{caption}
\usepackage{subcaption}

\usepackage{booktabs}
\usepackage{paralist}
\usepackage{soul}

\numberwithin{equation}{section}

\usepackage{algorithm}
\usepackage{algorithmicx}
\usepackage{algpseudocode}

\usepackage{cite}
\usepackage{multirow} 
\usepackage{enumitem} 
\setlist[enumerate]{label=(\roman*)}

\usepackage[colorlinks,citecolor=black,urlcolor=black]{hyperref}
	\hypersetup{allcolors=black}
\usepackage[nameinlink,noabbrev]{cleveref}

\theoremstyle{plain}
\newtheorem{theorem}{Theorem}[section]
\newtheorem{proposition}[theorem]{Proposition}
\newtheorem{lemma}[theorem]{Lemma}
\newtheorem{corollary}[theorem]{Corollary}
\newtheorem{remark}[theorem]{Remark}
\newtheorem{definition}[theorem]{Definition}

\DeclareMathOperator*{\argmin}{arg\,min} 
\DeclareMathOperator*{\esssup}{ess\,sup} 
\newcommand{\doublehookrightarrow}{%
	\mathrel{\mathrlap{{\mspace{4mu}\lhook}}{\hookrightarrow}}
}

\definecolor{mycolor1}{rgb}{0.00000,0.44700,0.74100}
\definecolor{mycolor2}{rgb}{0.85000,0.32500,0.09800}
\definecolor{mycolor3}{rgb}{0.92900,0.69400,0.12500}
\definecolor{mycolor4}{rgb}{0.46600,0.67400,0.18800}
\definecolor{mycolor5}{rgb}{0.49400,0.18400,0.55600}

\title[Minimum energy estimation for a cubic wave equation]{Local well-posedness of the minimum energy estimator for a defocusing cubic wave equation}
\author{Jesper Schr\"oder${}^\dagger$}
\address{${}^{\dagger}$  Institute of Mathematics MA\,{}4-4, Technische Universit\"at Berlin, Stra\ss e des 17.~Juni 136, 10623 Berlin, Germany}
\email{j.schroeder@tu-berlin.de}
\date{\today}

\begin{document}

\begin{abstract}
	This work is concerned with the \textit{minimum energy estimator} for a nonlinear hyperbolic partial differential equation. The \textit{Mortensen observer} -- originally introduced for the energy-optimal reconstruction of the state of nonlinear finite-dimensional systems -- is formulated for a disturbed cubic wave equation and the associated observer equation is derived. An in depth study of the associated optimal control problem and sensitivity analysis of the corresponding value function reveals that the energy optimal state estimator is well-defined. Deploying a classical fixed point argument we proceed to show that the observer equation is locally well-posed.
\end{abstract}

\maketitle

 
\section{Introduction}~\\
%
We consider the task of reconstructing the state of the system that is modeled by a cubic wave equation subject to disturbances given as
\begin{equation}\label{eq: DistSysIntro}
	\begin{alignedat}{2}
		\partial_{tt} {u} + {u}^3 &= \Delta {u} + v
		~~~~~ &&\text{ in } (0,T) \times \Omega, \\
		{u}(0) &= u_1 + \eta_1
		~~~~~&&\text{ in } \Omega,\\
		\partial_t {u}(0) &= u_2 + \eta_2
		~~~~~&&\text{ in } \Omega,
	\end{alignedat}
\end{equation}
where $\Omega \subset \mathbb{R}^3$ is the physical domain with sufficiently smooth boundary and the equation is complemented by homogeneous Dirichlet boundary conditions. The displacement of the wave at time $t \in [0,T]$ and position $x \in \Omega$ is denoted by $u(t,x)$. As usual $\Delta$ denotes the Laplacian on the spatial variable. We work in the usual \textit{energy space} $\mathcal{E} \coloneqq H^1_0(\Omega) \times L^2(\Omega)$ and the modeled initial displacement and velocity are given by $u_1 \in H^1_0(\Omega)$ and $u_2 \in L^2(\Omega)$, respectively. Finally, $v \in L^2(0,T;L^2(\Omega))$ denotes the \textit{unknown} deterministic disturbance in the system dynamics while $\eta = (\eta_1,\eta_2) \in\mathcal{E}$ represents the \textit{unknown} error in the initial state. The objective is the approximation of the displacement $u$ and velocity $\partial_t u$ based on a disturbed output that is available through measurements taken from the system. The output is modeled as 
\begin{equation}\label{eq: DistObsIntro}
	y_\mathrm{abs} = C \begin{bmatrix} u \\ \partial_t u \end{bmatrix} + \mu,
\end{equation}
where $C$ is a bounded, linear operator mapping from $H^1_0(\Omega) \times L^2(\Omega)$ to the Hilbert space $Y$. The \textit{unknown} disturbance in the output is represented by $\mu \in L^2(0,T;Y)$.

\subsection{Motivation and available literature}\label{subsec: lit}~\\
Semilinear wave equations with monomial nonlinearities of the form $ u \mapsto u^{2p + 1} $ frequently appear in mathematical physics \cite{Hei52,Joe59,Stu04} and are well-studied by the mathematical community. In particular, systems with physical dimension three have been of interest and authors differentiate the \textit{subcritical} ($p<2$), the \textit{critical} ($p=2$) and the \textit{supercritical} ($p>2$) case. Proofs of well-posedness of the equations in various settings are given in \cite{GinVe85,Joe61,Sat66} and \cite{Fei94,KalSZ16} study the wave equations with respect to attractors for the subcritical and the critical case. Further \cite{DehEtAl03} discusses stabilization and control of the subcritical case and \cite{KunM20} is concerned with optimal control of the critical case. We note that the positive sign of the nonlinearity is essential, as a negative sign is well-known to cause ill-posedness due to finite-time blow-ups, cf., \cite{Lio85}.
The study of designs for the state-reconstruction for disturbed dynamical systems has a rich history going back at least to Wiener's work \cite{Wie49} who interpreted the disturbances as random processes laying the foundation to the field of \textit{filtering theory}. In their pioneering works tackling time-discrete systems \cite{Kal60} and time-continuous systems \cite{KalB61} Kalman and Bucy introduce a filter optimal in squared expectation for linear dynamics. The demand to generalize the widely popular \textit{Kalman(-Bucy) filter} to nonlinear dynamics has inspired numerous designs such as the the \textit{extended Kalman filter} obtained by linearization of the dynamics along the filter and the more elaborate \textit{unscented Kalman filter} \cite{MeWaJu04}. While they are successfully applied in various settings, their derivation is purely heuristic and they lack the optimality property of the original Kalman filter. 

Taking a deterministic viewpoint, in \cite{Mor68} Mortensen formulates and formally reviews the so-called \textit{maximum-likelihood estimator}, today also known as the \textit{Mortensen observer}, which is constructed as the trajectory minimizing the squared energy of the unknown disturbances. As mentioned in the original work, applied to linear systems the \textit{Mortensen observer} reduces to the Kalman filter equations, see also \cite{Wil04} for the deterministic derivation of the Kalman filter. Therefore it can be interpreted as a truncation free generalization of the (deterministic interpretation of the) Kalman filter to nonlinear systems. Its dependence on the \textit{value function} corresponding to the optimal control problem associated with the energy minimization, however, inflicts the \textit{curse of dimensionality} making its numerical realization a challenging task.
Since its formal introduction the Mortensen observer has been studied for numerous classes of finite-dimensional systems, first in \cite{Fle97} under the assumption of smooth dynamics with globally bounded derivatives followed by \cite{Kre03} showing convergence of the observer assuming the dynamics to be globally Lipschitz. More recently a discrete-time version for affine dynamics was analyzed in \cite{Moi18} and \cite{BreKu21} proposes an approximation of the observer based on neural networks. The authors of \cite{ChaEtAl23} introduce and analyze a version for non-smooth systems. In particular a lack of equivalence with the stochastic interpretation is shown. In \cite{BreS24} a proof of well-posedness is given that drops the assumption of global Lipschitz continuity and in turn yields local results and \cite{BreKuSc23} proposes a numerical realization based on a polynomial approximation of the value function.

While the aforementioned literature exclusively discusses systems of finite dimension, generalizations to systems governed by partial differential equations are available for some of the concepts. We refer to the review article \cite{Cur75} on infinite-dimensional filtering that in particular discusses the Kalman filter and its well-posedness, see also \cite{Ben03}. Taking the deterministic perspective the authors of \cite{AfsGeMo23} were recently able to show well-posedness of the extended Kalman filter for a class of infinite-dimensional systems. While a wide array of observer designs for systems governed by nonlinear partial differential equations is available in the literature (e.g.,~\cite{AhmEtAl16,JadMeKu11,Rod21,ZhaWu20}),
to the best of the author's knowledge no rigorous analysis of the Mortensen observer applied to infinite-dimensional systems is available yet. 
%
\subsection{The Mortensen observer for a nonlinear wave equation}\label{subsec: Mortensen}~\\
%
In the following we introduce the Mortensen observer for the system described by \eqref{eq: DistSysIntro} and \eqref{eq: DistObsIntro} and present the formal derivation of the associated observer equation. Due to its formal nature it is conducted entirely analogous to the finite-dimensional case, see e.g. \cite{BreKu21}. For the sake of self-containedness we present a brief exposition of the concept.

In order to follow the structure found in \cite{BreKu21} we state the first order form of the wave equation and its associated output, which reads
\begin{equation}\label{eq: DistSysIntroCauch}
	\begin{aligned}
		\dot{w}(t) 
		&= \underbrace{ 
			\begin{bmatrix} 0 & \mathrm{Id} \\ \Delta &0 \end{bmatrix} w(t) 
			+ \begin{bmatrix} 0 \\ - w_1^3(t) \end{bmatrix}
		}_{\eqqcolon F(w(t))}
		+ B v(t),
		~~~~~ t \in (0,T),\\
		w(0) &= w_0 + \eta,\\
		y_\mathrm{abs}(t) &= C w(t) + \mu(t),
	\end{aligned}
\end{equation}
where
$w(t) = \begin{bmatrix} w_1(t) \\ w_2(t) \end{bmatrix} 
= \begin{bmatrix} u(t) \\ \partial_t u(t) \end{bmatrix}$,
$B v = \begin{bmatrix} 0 \\ v \end{bmatrix}$,
$w_0 = \begin{bmatrix} u_1 \\ u_2 \end{bmatrix}$,
and
$\eta = \begin{bmatrix} \eta_1 \\ \eta_2 \end{bmatrix}$.
The construction of the state estimate $\widehat{w}(t)$ at time $t \in (0,T]$ is based on the optimal control problem 
\begin{align}
	\begin{split}
		\min_{w,v} J(w,v;t)
		&= \frac{1}{2} \left\Vert w(0) - w_0 \right\Vert_\mathcal{E}^2
		+ \frac{1}{2} \int_0^t \Vert v (s) \Vert_{L^2(\Omega)}^2 
		+ \alpha \left\Vert y_\mathrm{abs}(s) - C w(s) \right\Vert_{Y}^2 \, \mathrm{d}s,
		\label{eq: Jintro}
	\end{split}\\
	\begin{split}
		\text{s.t.}~~~\dot{w}(s) &= F(w(s)) + B v(s),
		~~~~~s \in (0,t),
		\label{eq: StateIntro}\\
		w(t) &= \xi, 
	\end{split}
\end{align}
with $y_\mathrm{abs}$ obtained via measurement and fixed final value $\xi$. 
Intuitively speaking, \eqref{eq: Jintro}-\eqref{eq: StateIntro} formalizes the task of finding energy minimal disturbances $v$, $\eta$, and $\mu$ that fit the output $y$ measured until time $t$ and the fixed state $\xi$ at time $t$. The associated value function
\begin{equation*}
	\mathcal{V}(t,\xi;y_\mathrm{abs}) = \inf J(w,v;t)
	~~~~~\text{subject to}~ \eqref{eq: StateIntro}
\end{equation*} 
represents the \textit{minimal amount of energy} required to enforce the state $\xi$ at time $t$ and is formally associated with the \textit{Hamilton-Jacobi-Bellman equation}
\begin{equation}\label{eq: HJBIntro}
	\begin{aligned}
		\partial_t \mathcal{V}(t,\xi)
		= - \left( D_\xi \mathcal{V}(t,\xi), F(\xi) \right)_\mathcal{E}
		- \frac{1}{2} \Vert B^* D_\xi \mathcal{V}(t,\xi) \Vert_{L^2(\Omega)}^2
		+\frac{\alpha}{2} \Vert y_\mathrm{abs}(t) - C\xi \Vert_Y^2.
	\end{aligned}
\end{equation}
The Mortensen observer is defined pointwise in time as the energy minimizing state, i.e.,
\begin{equation}
	\widehat{w}(t) \coloneqq \argmin_{\xi \in \mathcal{E}} \mathcal{V}(t,\xi;y_\mathrm{abs}),~~~~~ t \in [0,T].
\end{equation}
For a derivation of a governing equation first note that it immediately follows
\begin{equation*}
	D_\xi \mathcal{V}(t,\widehat{w}(t)) = 0 ~~~~~\forall t \in [0,T].
\end{equation*}
Taking a time derivative and applying the chain rule yields
\begin{equation}\label{eq: ObsDer}
	\partial_t D_\xi \mathcal{V}(t,\widehat{w}(t))
	= - D_{\xi\xi}^2 \mathcal{V}(t,\widehat{w}(t)) \, \dot{\widehat{w}}(t).
\end{equation}
Taking a derivative with respect to $\xi$ along $\widehat{w}(t)$ of the HJB \eqref{eq: HJBIntro} shows
\begin{equation*}
	D_\xi \partial_t \mathcal{V}(t,\widehat{w}(t))
	=
	- D_{\xi\xi}^2 \mathcal{V}(t,\widehat{w}(t)) \, F(\widehat{w}(t))
	- \alpha C^* \left( y_\mathrm{abs}(t) - C \widehat{w}(t) \right).
\end{equation*}
Switching the order of differentiation on the left hand side and inserting \eqref{eq: ObsDer} leads to
\begin{equation*}
	D_{\xi\xi}^2 \mathcal{V}(t,\widehat{w}(t)) \, \dot{\widehat{w}}(t)
	= D_{\xi\xi}^2 \mathcal{V}(t,\widehat{w}(t)) \, F(\widehat{w}(t))
	+ \alpha C^* \left( y_\mathrm{abs}(t) - C \widehat{w}(t) \right).
\end{equation*}
Finally by assuming $D_{\xi\xi}^2 \mathcal{V}(t,\widehat{w}(t))$ is an isomorphism and applying its inverse we arrive at the Mortensen observer equation
\begin{equation}\label{eq: MorForm}
	\dot{\widehat{w}}(t)
	= F(\widehat{w}(t))
	+ \alpha D_{\xi\xi}^2 \mathcal{V}(t,\widehat{w}(t))^{-1} C^* 
	\left( y_\mathrm{abs}(t) - C \widehat{w}(t) \right).
\end{equation}
\begin{remark}
	We comment on several critical points regarding the optimal control problem \eqref{eq: Jintro}-\eqref{eq: StateIntro}.	
	\begin{enumerate}[label=(\roman*)] 
		\item Even though \eqref{eq: Jintro}-\eqref{eq: StateIntro} is formulated as an optimal control problem, we do not have any means to influence the system. Here, the optimal control formulation is to be understood as a tool to reconstruct energy minimal disturbances only. Nonetheless, since the technical analysis of the problem is tackled using strategies from optimal control, throughout this work $v$ will be referred to as a control.
		\item The dynamical system \eqref{eq: DistSysIntroCauch} as well as the observer equation \eqref{eq: MorForm} are given an initial condition and evolve forward in time. The analysis, however, is in large parts concerned with the state equation \eqref{eq: StateIntro} which evolves backward in time. While this poses a major challenge for parabolic dynamics such as diffusion, it does not pose any issue for the undamped hyperbolic dynamics under consideration, cf., \Cref{lem: skewAdj} below. 
	\end{enumerate}
\end{remark}
%
%
\subsection{Contributions of this work}~\\
%
The foregoing formal derivation illustrates the most pressing questions regarding well-posedness of the concept. In the following we summarize our contributions:
\begin{itemize}
	\item[(i)] We perform a rigorous analysis of the optimal control problem in order to show local $C^\infty$ regularity of the associated value function $\mathcal{V}$ with respect to space and measurement, cf., \Cref{cor: ValFunReg}. This in particular implies the existence of its second derivative $D_{\xi\xi}^2 \mathcal{V}$.
	\item[(ii)] By analyzing the linearized optimal control we conduct the proof of \Cref{prop: D2Coerc} stating that under assumptions on the data $D_{\xi\xi}^2 \mathcal{V}$ is a coercive bilinear form, implying that the associated linear mapping is an isomorphism and hence invertible and further that $\mathcal{V}$ is locally strictly coercive.
	\item[(iii)] By an application of the implicit function theorem we show existence of a unique root of $D_\xi \mathcal{V}(t,\cdot)$, ensuring the well-definedness of the Mortensen observer as a pointwise in time minimizer of the squared energy of disturbances, cf., \Cref{cor: MorWD}. 
	\item[(iv)] Finally we utilize existence and boundedness of $D_{\xi\xi}^2 \mathcal{V}^{-1}$ to apply a fixed point argument and show existence of a locally unique mild solution to the observer equation, cf., \Cref{thm: ObsEqWP}.
\end{itemize}
In large parts the paper follows the structure of \cite{BreS24} where we investigated these issues for finite-dimensional systems. In the setting considered in this work, however, solutions of the state equation are less regular and therefore alternative techniques are deployed for the analysis of the second derivative of the value function, cf., \Cref{sec: CharSecDer} and well-posedness of the observer equation, cf., \Cref{subsec: MorObsEq}. In particular, the analytical arguments deployed here do not rely on a Hamilton-Jacobi-Bellman equation.
%
\subsection{Notation}~\\
%
We denote by $\Omega$ a bounded domain from $\mathbb{R}^3$ with sufficiently smooth boundary $\Gamma$. By $L^p(\Omega)$, $1 \leq p \leq \infty$ and $H^1(\Omega)$ we denote the Lebesgue spaces and the Sobolev space of functions with a first weak derivative in $L^2(\Omega)$, respectively. Further $H^1_0(\Omega) \subset H^1(\Omega)$ denotes the subspace of all elements with zero trace on $\Gamma$. The associated norms are $\Vert w \Vert_{L^p(\Omega)}^p = \int_\Omega \Vert w(x) \Vert^p \,\mathrm{d}x $, $\Vert w \Vert_{L^\infty(\Omega)} = \esssup_{x \in \Omega} \Vert x \Vert $, and $\Vert w \Vert_{H^1_0(\Omega)} = \Vert \nabla w \Vert_{L^2(\Omega)}$. The space $H^{-1}(\Omega)$ is the usual Sobolev space equipped with its natural norm. 

Let $N \in \mathbb{N}$ and $X_i$, $i = 1,...,N$, $X$ and $Y$ be Banach spaces. The Cartesian product of $X_i$ is denoted by $X_1 \times ... 
\times X_N$. Unless mentioned otherwise its norm is $\Vert (x_1,...,x_N) \Vert_{X_1 \times... \times X_N} = \max\limits_{i=1,...,N} \Vert x_i \Vert_{X_i}$. The space of all multilinear forms mapping from $X_1 \times ... \times X_N$ to $Y$ is denoted as $\mathcal{L}(X_1,...,X_n;Y)$ and is equipped with the standard norm. The space $\mathcal{L}(X;X)$ is denoted by $\mathcal{L}(X)$. The identity operator on $X$ is denoted by $\mathrm{Id}_X$, where the subscript is dropped whenever the space is clear from context. For an element $x \in X$ and a real number $\epsilon > 0$ the open ball of radius $\epsilon$ around $x$ is denoted by $\mathcal{U}_\epsilon(x)$. The weak convergence of a sequence $x_k$ to some $x$ is denoted as $x_k \rightharpoonup x$, for $k \to \infty$. 

For a function $f \colon X \to Y$ its Fr\'{e}chet derivative at $x \in X$ applied to $z \in X$ is denoted by $Df(x)[z]$. For a function $g \colon X_1 \times X_2 \to Y$ the partial Fr\'{e}chet derivative with respect to the first variable at $(x_1,x_2) \in X_1 \times X_2 $ applied to $z \in X_1$ is denoted by $D_{x_1}g(x_1,x_2)[z]$.  Higher order and mixed partial derivatives are denoted with appropriate indices, e.g., $D^2_{x_1 x_2}g(x_1,x_2)[z_2,z_1]$ denotes the second order partial derivative at $(x_1,x_2)$ applied to $z_2 \in X_2$ and $z_1 \in X_1$. A combination of full and partial derivatives is denoted as $D_{x_1} D g(x_1,x_2) [(z_1,z_2),z_1^\prime]$ an is to be understood as the partial derivative with respect to $x_1$ in direction $z_1^\prime$ of the full derivative in direction $(z_1,z_2) \in X_1 \times X_2$.

We denote the dual space of $X$ by $X^*$ and for $f \in X^*$ and $x \in X$ we denote their dual pairing as $\langle f, x \rangle_{X^*,X} \coloneqq f(x)$. For $A \in \mathcal{L}(X;Y)$ the adjoint operator is denoted as $A^* \in \mathcal{L}(Y^*;X^*)$. A continuous embedding of $X$ into $Y$ is denoted by $X \hookrightarrow Y$ and if the embedding is compact we write $X \doublehookrightarrow Y$.

For $1 \leq p \leq \infty $ we denote by $L^p(0,T;X)$ the Lebesgue spaces, where integrals are understood in the Bochner sense and the norm is defined analogous to the real case. Furthermore $C([0,T];X)$ and $C^k([0,T];X)$ denote the spaces of functions $f \colon [0,T] \to X$ that are continuous and $k$ times continuously differentiable, respectively.

Throughout this work we use $c$ as a generic constant.
%
\section{Analytical framework}\label{sec: Bas}~\\
%
This section presents the basic notions required for our analysis. In particular appropriate solution concepts are presented and well-posedness of the wave equation is discussed.

We commence by stating a lemma essential for treating the cubic wave equation in three spatial dimensions.
Only because the spatial dimension is set to three we have the continuous embedding $H^1_0(\Omega) \hookrightarrow L^6(\Omega)$, see \cite[P.~II, Sec.~5.6, Thm.~2]{Eva10}, and throughout this work denote the associated constant by $c_\mathrm{em} > 0$.
An application of H\"olders inequality yields the following estimate.
\begin{lemma}\label{lem: cubicInL2}
	The product of three functions $w_1,w_2,w_3 \in H^1_0(\Omega)$ lies in $ L^2(\Omega)$ and it holds
	\begin{equation*}
		\Vert w_1 \, w_2 \, w_3 \Vert_{L^2(\Omega)}
		\leq c_\mathrm{em}^3 \, \Vert w_1 \Vert_{H^1_0(\Omega)}
		\Vert w_2 \Vert_{H^1_0(\Omega)}
		\Vert w_3 \Vert_{H^1_0(\Omega)}.
	\end{equation*}
\end{lemma}
\begin{proof}
	The assertion is a direct consequence of the embedding $H^1_0(\Omega) \hookrightarrow L^6(\Omega)$ and H\"older's inequality.
\end{proof}
%
%
\subsection{The forward equation}~\\
%
In this subsection we discuss the solution theory concerned with the forward equation \eqref{eq: DistSysIntro}.
For the sake of readability for $t \in(0,T]$ we introduce the notation $\mathcal{L}_t = L^2(0,t;L^2(\Omega))$.
To simplify the presentation for the moment we set $\eta = 0$ and the equation reads 
\begin{equation}\label{eq: FWE}\tag{FWE}
	\begin{alignedat}{2}
		\partial_{tt} u + u^3 &= \Delta u + v~~~~~ &&\text{ in } (0,T) \times \Omega, \\
		(u(0),\partial_t u(0)) &= (u_1,u_2)~~~
		&&\text{ in } \Omega.
	\end{alignedat}
\end{equation}
As above we assume $v \in \mathcal{L}_t$.
In the following we briefly introduce the analytical tools required to formulate the concept of a \textit{mild solution} to \eqref{eq: FWE}. The energy space $\mathcal{E} = H^1_0(\Omega) \times L^2(\Omega)$ is equipped with the scalar product 
\begin{equation*}
	\left( \begin{bmatrix} w_1 \\ w_2 \end{bmatrix} , \begin{bmatrix} p_1 \\ p_2 \end{bmatrix} \right)_\mathcal{E} 
	\coloneqq (w_1,p_1)_{H^1_0(\Omega)} + (w_2,p_2)_{L^2(\Omega)}
\end{equation*}
and forms a Hilbert space. In order to formulate the wave equation as a first order problem in $\mathcal{E}$ consider the linear, unbounded operator in $\mathcal{E}$ defined via
\begin{equation}\label{eq: mathcalA}
	\begin{aligned}
		\mathcal{D}(\mathcal{A}) = \mathcal{D}(\Delta) \times H^1_0(\Omega),
		~~~~~
		\mathcal{A} \begin{bmatrix} w_1 \\ w_2 \end{bmatrix} 
		= \begin{bmatrix} 0 & I \\ \Delta & 0 \end{bmatrix}
		\begin{bmatrix} w_1 \\ w_2 \end{bmatrix}
		= \begin{bmatrix} w_2 \\ \Delta w_1 \end{bmatrix},
	\end{aligned}
\end{equation}
where the Laplacian $\Delta$ is understood as a linear unbounded operator in $L^2(\Omega)$ defined via
\begin{equation*}
	\begin{aligned}
		\mathcal{D}(\Delta) 
		&= \left\{ w \in H^1_0(\Omega) \colon \exists f \in L^2(\Omega) \colon (\nabla w, \nabla \varphi)_{L^2(\Omega)} = - (f,\varphi)_{L^2(\Omega)} \forall \varphi \in H^1_0(\Omega) \right\}\\
		\Delta w &= f.
	\end{aligned}
\end{equation*}
We set $w_0 = \begin{bmatrix} u_1 \\ u_2 \end{bmatrix}$ and formally rewrite \eqref{eq: FWE} as the \textit{Cauchy problem}
\begin{equation}\label{eq: AbstrMod}
	\begin{aligned}
		\frac{\mathrm{d}}{\mathrm{d}t} w(t) \,
		&= \mathcal{A} \, w (t)
		+ \begin{bmatrix} 0 \\ - w_1^3(t) + v(t) \end{bmatrix}~\text{ for } t \in (0,T),~~~
		w(0) = w_0.
	\end{aligned}
\end{equation}
Note that any sufficiently smooth classical solution $u$ of \eqref{eq: FWE} yields a tuple $w = \begin{bmatrix} u \\ \partial_t u \end{bmatrix}$ satisfying \eqref{eq: AbstrMod}.
Before defining a proper solution concept associated with the Cauchy problem we state some properties of $\mathcal{A}$. 
\begin{lemma}\label{lem: skewAdj}
	The linear unbounded operator $\mathcal{A}$ defined in \eqref{eq: mathcalA} is skew-adjoint, i.e., $\mathcal{D}(\mathcal{A}^*) = \mathcal{D}(\mathcal{A})$ and for all $w \in \mathcal{D}(\mathcal{A})$ it holds $\mathcal{A}^* w = - \mathcal{A} w$.	Further $\mathcal{A}$ generates a unitary group on $\mathcal{E}$ denoted by $t \mapsto e^{\mathcal{A}t}$, i.e., for all $t \in \mathbb{R}$ and $w \in \mathcal{E}$ we have
	\begin{equation*}
		e^{\mathcal{A}t} \in \mathcal{L}(\mathcal{E}),
		~~~~~ \Vert e^{\mathcal{A}t} w \Vert_\mathcal{E} = \Vert w \Vert_\mathcal{E}.
	\end{equation*}
\end{lemma}
\begin{proof}
	An application of integration by parts shows the skew-adjointness. According to \cite[Rem.~6.2.6]{JacZw12} any skew-adjoint operator generates a unitary group and the assertion is shown.
\end{proof}
In particular $\mathcal{A}$ generates a strongly continuous semi-group denoted by $e^{\mathcal{A}t}$ allowing the following definition of \textit{mild solutions}. For convenience we denote the space in which the solutions lie by $\mathcal{C}_t \coloneqq C([0,t];\mathcal{E})$.
\begin{definition}\label{def: MildSolForw}
	A function $w$ is called a mild solution to \eqref{eq: FWE} on $[0,T]$ if $w \in \mathcal{C}_T$ and for $t \in [0,T]$ it holds
	\begin{equation}\label{eq: mildSolu}
		w(t)
		= 
		e^{\mathcal{A}t} w_0
		+ \int_0^t e^{\mathcal{A}(t-s)} \begin{bmatrix} 0 \\ -w_1^3(s) + v(s) \end{bmatrix} \, \mathrm{d} s.
	\end{equation}
\end{definition}
\begin{remark}\label{rem: mildSol}
	\begin{enumerate}[label=(\roman*)] 
		\item We note that the integrand in \eqref{eq: mildSolu} is not trivially well defined. Only due to $\Omega \subset \mathbb{R}^3$ and \Cref{lem: cubicInL2} we have $w_1^3(s) \in L^2(\Omega)$. This does not need to hold for other exponents or dimensions of $\Omega$.
		\item As mentioned above a classical solution $u$ of \eqref{eq: FWE} is not a mild solution according to \Cref{def: MildSolForw}. Other authors avoid this issue using a different notion of mild solution, see e.g., \cite[Def.~2.1]{KunM20}. The definition stated in this work instead is in line with the notion of mild solutions for general Cauchy problems, cf., \cite[Def.~5.1.4]{CurZw20} and is better suited for our analysis.
		\item Any mild solution $w = \begin{bmatrix} w_1 \\ w_2 \end{bmatrix}$ satisfies $w_2 = \partial_t w_1$. This is due to the structure of $\mathcal{A}$ and can be seen by noting that $w$ solves a linear problem with a state dependent inhomogeneity and considering \cite[Prop.~3.1.16]{AreEtAl13}.
	\end{enumerate}
\end{remark}
We additionally introduce the notion of \textit{weak solutions}.
\begin{definition}\label{def: WeakSol}
	A function $u \in L^\infty(0,T;H^1_0(\Omega))$ with $ \partial_t u \in L^\infty(0,T;L^2(\Omega))$ is called a weak solution to \eqref{eq: FWE} if
	$u(0) = u_1$, $\partial_t u(0) = u_2$
	and additionally for all $\varphi \in C_0^\infty((0,T)\times \Omega)$ it holds
	\begin{equation}
		\begin{aligned}
			-\int_0^T (\partial_t u, \partial_t \varphi)_{L^2(\Omega)} \, \mathrm{d}t
			+ \int_0^T (u^3,\varphi)_{L^2(\Omega)} \, \mathrm{d}t
			= - \int_0^T (\nabla u,\nabla \varphi)_{L^2(\Omega)} \, \mathrm{d}t
			+ \int_0^T (v,\varphi)_{L^2(\Omega)} \, \mathrm{d}t.
		\end{aligned}
	\end{equation}
	Note that we have $u^3 \in L^\infty(0,T;L^2(\Omega))$, and hence $ \partial_{tt} u \in L^\infty(0,T;H^{-1}(\Omega))$. Therefore $u \in C([0,T];L^2(\Omega))$ and $ \partial_t u \in C([0,T];H^{-1}(\Omega))$ and the initial conditions are well-defined.	
\end{definition}
The concept of a weak solution already makes sense for initial values from $L^2(\Omega) \times H^{-1}(\Omega)$ and right hand sides $v \in L^1(0,T;H^{-1}(\Omega))$, leading to less regular solutions. For our purpose we assume more regular data leading to equivalence of mild and weak solutions. While the correspondence of mild and weak solutions is well-known in the general linear setting, see e.g. \cite[Thm.~5.1.10]{CurZw20}, in our case this issue requires some extra care. 
\begin{lemma}\label{lem: MildEqWeak}
	A function $w= \begin{bmatrix} w_1 \\ w_2 \end{bmatrix} = \begin{bmatrix} w_1 \\ \partial_t w_1 \end{bmatrix}$ is a mild solution to \eqref{eq: FWE} if and only if $w \in \mathcal{C}_T$ and $w_1$ is a weak solution.
\end{lemma}
The proof can be done analogous to the proof of \cite[Lem.~2.6]{KunM20} and is omitted here. Note that at this point there might be discontinuous weak solutions that do not yield any mild solution. This possibility is ruled out by the following result on well-posedness of weak solutions. 
\begin{proposition}\label{prop: WaveWP}
	For any $u_1 \in H^1_0(\Omega)$, $u_2 \in L^2(\Omega)$ and $v \in \mathcal{L}_t$ there exists exactly one weak solution $u \in L^\infty(0,T;\mathcal{E})$ to \eqref{eq: FWE}. Further $u \in C([0,T];H^1_0(\Omega))$ and $\partial_t u \in C([0,T];L^2(\Omega))$. 
\end{proposition}
The proof can be conducted by a well-known \textit{Galerkin approximation}. For a thorough presentation we refer to classical results by Lions. Existence and uniqueness of a weak solution $u$ are presented as Th\'eor\`eme 1.1 and Th\'eor\`eme 1.2 in \cite{Lio69}, respectively. The asserted regularity follows by noting that $u$ is the weak solution of a linear wave equation with inhomogeneity $f = v - u^3 \in \mathcal{L}_t$ and considering \cite[Ch.~3, Thm.~8.2]{LioM72}.

Well-posedness of mild solutions follows immediately. 
\begin{corollary}\label{cor: WE_WP}
	For $u_1 \in H^1_0(\Omega)$, $u_2 \in L^2(\Omega)$ and $v \in \mathcal{L}_t$ there exists exactly one mild solution $w \in \mathcal{C}_T$ to \eqref{eq: FWE}.
\end{corollary}

From here on the term solution will refer to a mild solution unless stated otherwise. We conclude the section by defining the \textit{nominal trajectory}. 
The model equation in its strong form is given by 
\eqref{eq: DistSysIntro} with $v = 0$ and $\eta = 0$, i.e.,
\begin{equation*}
	\begin{alignedat}{2}
		\partial_{tt} \Tilde{u} + \Tilde{u}^3 &= \Delta \Tilde{u}~~~~~ &&\text{ in } (0,T) \times \Omega, \\
		(\Tilde{u}(0),\partial_t \Tilde{u}(0)) &= (u_1,u_2)~~~
		&&\text{ in } \Omega.
	\end{alignedat}
\end{equation*}
According to \Cref{cor: WE_WP} there exists a unique mild solution which throughout this chapter we denote by $\Tilde{w} \in \mathcal{C}_T$ and refer to it as the \textit{nominal trajectory} or \textit{model trajectory}.
%
\subsection{The backward state equation}~\\
%
This section introduces solution concepts for the backward state equation of the optimal control problem and states that results from the previous section transfer. Its strong form reads 
\begin{equation}\label{eq: SWE}\tag{SWE}
	\begin{alignedat}{2}
		\partial_{tt} u + u^3 &= \Delta u + v &&\text{ in } (0,t) \times \Omega, \\
		(u(t),\partial_t u(t)) &= (\Tilde{w}_1(t),\partial_t \Tilde{w}_1(t)) + \xi ~~~&&\text{ in } \Omega.
	\end{alignedat}
\end{equation}
Unlike the model it is not considered on $[0,T]$ but on $[0,t]$, where $t \in (0,T]$ is fixed. The disturbances $v \in \mathcal{L}_t$ is fixed and the final value consists of the model trajectory at time $t$ and a fixed disturbance $\xi \in \mathcal{E}$. The equation is completed with homogeneous Dirichlet boundary conditions. The following definition of mild solutions for the backward state equation relies on the fact that $\mathcal{A}$ generates a group.
\begin{definition}\label{def: MildSolBackw}
	A function $w$ is called a mild solution to \eqref{eq: SWE} if $w \in \mathcal{C}_t$ and for all $s \in (0,t]$ it holds
	\begin{equation}\label{eq: mildSol}
		w (s)
		= e^{-\mathcal{A}(t-s)}
		\left( \Tilde{w}(t) + \xi \right)
		-
		\int_s^t e^{-\mathcal{A}(\tau-s)} \begin{bmatrix} 0 \\ -w_1^3(\tau) + v(\tau) \end{bmatrix} \, \mathrm{d}\tau.
	\end{equation}
\end{definition}
\begin{remark}\label{rem: DefSol}{}~
	To illustrate the relationship between the forward and the backward formulation and for later use we note that any solution $w$ of \eqref{eq: FWE} can be viewed as the solution of a backward problem on $(0,t)$ with final value $w_t = w(t)$, i.e.,
	\begin{equation}\label{eq: mildSoluBack}
		w (s) 
		= e^{-\mathcal{A}(t-s)} w_t 
		- \int_s^t e^{-\mathcal{A}(\tau-s)} \begin{bmatrix} 0 \\ - w_1^3(\tau) + v(\tau) \end{bmatrix} \,\mathrm{d}\tau,
		~~~\forall s \in [0,t].
	\end{equation}
\end{remark}
Weak solutions for the backward are defined entirely analogous to the forward case.
\begin{definition}
	A function $u \in L^\infty(0,t;H^1_0(\Omega))$ with $ \partial_t u \in L^\infty(0,t;L^2(\Omega))$ is called a weak solution to \eqref{eq: SWE} if
	$(u(t),\partial_t u(t)) = \Tilde{w} (t) + \xi $
	and additionally for all $\varphi \in C_0^\infty((0,t)\times \Omega)$ it holds
	\begin{equation}\label{eq: weakFormula}
		\begin{aligned}
			-\int_0^t (\partial_t u, \partial_t \varphi)_{L^2(\Omega)} \, \mathrm{d}s
			+ \int_0^t (u^3,\varphi)_{L^2(\Omega)} \, \mathrm{d}s
			= - \int_0^t (\nabla u,\nabla \varphi)_{L^2(\Omega)} \, \mathrm{d}s
			+ \int_0^t (v,\varphi)_{L^2(\Omega)} \, \mathrm{d}s.
		\end{aligned}
	\end{equation}
\end{definition}
The results presented in \Cref{rem: mildSol} (iii) and \Cref{lem: MildEqWeak} for the forward equation can be proven analogously for the backward case. Reversing the time in \eqref{eq: weakFormula} via $\overleftarrow{u}(s) = u(t-s)$ it is shown that $u$ is a weak backward solution if and only if $\overleftarrow{u}$ is a weak solution to a forward problem with initial condition $\overleftarrow{u}(0) = \Tilde{w}_1 + \xi_1$ and $\partial_t \overleftarrow{u}(0) = - \Tilde{w}_1 - \xi_1$ and the well-posedness transfers.   
\begin{corollary}\label{cor: BackwWP}
	Let $t \in (0,T]$, $\xi \in \mathcal{E}$ and $v \in \mathcal{L}_t$. Then there exists a unique weak solution $u$ to \eqref{eq: SWE}. It holds $u \in C([0,t];H^1_0(\Omega))$ and $\partial_t u \in C([0,t];L^2(\Omega))$. Further \eqref{eq: SWE} admits a unique mild solutions and it is given by $w = \begin{bmatrix} u \\ \partial_t u \end{bmatrix}$.
\end{corollary}
%

%
%
\section{Optimal control problem}\label{sec: OCP}
In this section we give a rigorous formulation of the optimal control problem analyze it with respect to existence of solutions, and the first order optimality condition. 
The cost functional of the optimal control problem for the construction of the Mortensen observer for a fixed $t \in (0,T]$ is given via
\begin{equation}\label{eq: OCP}\tag{OCP}
	\begin{aligned}
		&J(\cdot,\cdot \, ;t,y) \colon \mathcal{C}_t \times \mathcal{L}_t \rightarrow \mathbb{R}\\
		&\left( w , v \right)
		\mapsto 
		\frac{1}{2} \left\Vert w(0) - w_0 \right\Vert_{\mathcal{E}}^2
		+ \frac{1}{2} \Vert v \Vert_{\mathcal{L}_t}^2 
		+ \frac{\alpha}{2} \left\Vert y - C \left( w - \Tilde{w} \right) \right\Vert_{\mathcal{Y}_t}^2,
	\end{aligned}
\end{equation}
subject to
\begin{equation*}
	e^t(w,v,\xi) = 0,
\end{equation*}
where the constraint $e^t \colon \mathcal{C}_t \times \mathcal{L}_t \times \mathcal{E} \rightarrow \mathcal{C}_t$ is defined as
\begin{equation}\label{eq: ConstrMapping}
	e^t(w,v,\xi) (s)
	= e^{-\mathcal{A}(t-s)} \left( \Tilde{w}(t) + \xi \right)
	- \int_s^t e^{-\mathcal{A}(\tau-s)} 
	\begin{bmatrix}
		0 \\ - w_1^3 + v
	\end{bmatrix}
	\, \mathrm{d} \tau 
	\, - \, w(s).
\end{equation}
For $t \in (0,T]$ we denote $\mathcal{Y}_t \coloneqq L^2(0,t;Y)$ and $y \in \mathcal{Y}_T$ is the measured output relative to the modeled output, i.e.,
\begin{equation*}
	y(s) = y_{\mathrm{abs}}(s) - C \Tilde{w}(s),
\end{equation*}
where $y_\mathrm{abs} \in \mathcal{Y_T}$ is the given absolute disturbed measurement. Finally $\xi \in \mathcal{E}$ is a fixed disturbance in the final value. For convenience we denote $c_C= \Vert C \Vert_{\mathcal{L}(\mathcal{E};Y)} = \Vert C^* \Vert_{\mathcal{L}(Y^*;\mathcal{E}^*)}$.

Note that \eqref{eq: OCP} is trivial for any $t \in (0,T]$ in case $(\xi,y) = (0,0)$. Here the minimizing pair is given by the model and $J(\Tilde{w},0;t,0) = 0$. 
\begin{remark}\label{rem: ShiftOCP}
	We stress that this formulation differs from the original exposition in \eqref{eq: Jintro}-\eqref{eq: StateIntro} in two subtle aspects. The tracking term is transformed equivalently to depend on $y$ instead of $y_\mathrm{abs}$. Further, in the original formulation $\xi$ represents the final state of the system, while here it describes its \textit{difference from the nominal trajectory}. Denoting the value function associated with \eqref{eq: OCP} by $\mathcal{V}$ and the one from the original formulation \eqref{eq: Jintro}-\eqref{eq: StateIntro} by $\mathcal{V}_\mathrm{orig}$ we have
	\begin{equation*}
		\mathcal{V}_\mathrm{orig}(t,\xi + \Tilde{w}(t),y_\mathrm{abs})
		= \mathcal{V}(t,\xi,y).
	\end{equation*}
	In a sense, the optimal control problem was shifted by $\Tilde{w}$. As a result our local analysis is performed on more convenient neighborhoods of zero, rather than neighborhoods of the nominal trajectory and its associated output. 
\end{remark}
%
\subsection{The data-to-state operator}~\\
%
Based on \Cref{cor: BackwWP} we define the data-to-state operator and proceed to show smoothness and uniform boundedness of said operator.
\begin{definition}
	The control-to-state operator 
	\begin{equation*}
		\mathcal{S}^t \colon \mathcal{E} \times \mathcal{L}_t \longrightarrow \mathcal{C}_t
	\end{equation*}
	maps $(v,\xi)$ to the solution $w$ of \eqref{eq: SWE} corresponding to $\xi$ and $v$.
\end{definition}
In the following we discuss the differentiability of $\mathcal{S}^t$ and characterize its derivatives using the implicit function theorem. For the benefit of readability here we abbreviate $w = \mathcal{S}^t(v,\xi)$. Note that even though we only give explicit formulas for the partial derivatives, formulas for the full derivative follow with identities of the form $D \mathcal{S}^t(v,\xi) [(u,\eta)] = D_v \mathcal{S}^t(v,\xi) [u] + D_\xi \mathcal{S}^t(v,\xi) [\eta]$.

\begin{lemma}\label{lem: SmoothS}
	The data-to-state operator is of class $C^\infty$ on its domain and the partial derivatives can be characterized as unique solutions of linear problems.
	Let $\xi$, $\eta$, $\psi \in \mathcal{E}$ and $v$, $u$, $h \in \mathcal{L}_t$.\\
	Then $D_v \mathcal{S}^t(v,\xi) [u]$ solves
	\begin{equation}\label{eq: Sder1}
		z(s)
		= 
		-\int_s^t e^{-\mathcal{A}(\tau-s)} \begin{bmatrix} 0 \\ - 3 w_1^2(\tau) z_1(\tau) + u(\tau) \end{bmatrix} \, \mathrm{d} \tau ~~~\forall s \in [0,t]
	\end{equation}
	and $D_\xi \mathcal{S}^t(v,\xi) [\eta]$ solves
	\begin{equation}\label{eq: Sder2}
		z(s)
		= 
		e^{-\mathcal{A}(t-s)} \eta
		-
		\int_s^t e^{-\mathcal{A}(\tau-s)} \begin{bmatrix} 0 \\- 3 	w_1^2(\tau) z_1(\tau) \end{bmatrix} \, \mathrm{d} \tau~~~\forall s \in [0,t].
	\end{equation}
	For partial derivatives of second order consider the formula 
	\begin{equation}\label{eq: Sder3}
		z(s)
		= 
		-\int_s^t e^{-\mathcal{A}(\tau-s)} \begin{bmatrix} 0 \\ 
			-3 w_1^2(\tau) z_1(\tau) 
			-6 w_1(\tau) \, p_1(\tau) \, q_1(\tau) \end{bmatrix} \, \mathrm{d} \tau.
	\end{equation}
	Then $D_{vv}^2 \mathcal{S}^t(v,\xi)[u,h]$ solves \eqref{eq: Sder3} if we set $p = D_v \mathcal{S}^t(v,\xi)[h] $ and $q = D_v \mathcal{S}^t(v,\xi)[u]$, $D_{\xi\xi}^2 \mathcal{S}^t(v,\xi)[\eta,\psi]$ solves \eqref{eq: Sder3} with $p = D_\xi \mathcal{S}^t(v,\xi)[\psi]$ and $q = D_\xi \mathcal{S}^t(v,\xi)[\eta]$, and finally $D_{\xi v}^2 \mathcal{S}^t(v,\xi)[u,\eta]$ solves \eqref{eq: Sder3} with $p = D_v \mathcal{S}^t(v,\xi) [u]$ and $q = D_\xi \mathcal{S}^t(v,\xi)[\eta]$.
\end{lemma}
\begin{proof}
	We apply the implicit function theorem to the mapping $e^t$.
	By an application of \cite[Thm.~7]{GoldKT92} it can be shown that the Nemytskii operator induced by $w \mapsto w^3$ mapping $L^\infty(0,t;L^6(\Omega))$ to $L^\infty(0,t;L^2(\Omega))$ is of class $C^\infty$. All remaining terms are (affine) linear bounded operators. Therefore $e^t$ as defined in \eqref{eq: ConstrMapping} is of class $C^\infty$.
	To show the required bijectivity of the partial derivative note that for $z,f \in \mathcal{C}_t$ it holds
	%
	$D_w e^t(\mathcal{S}^t(v,\xi),v,\xi) [z] = f$
	%
	if and only if for all $s \in [0,t]$ it holds
	\begin{equation}\label{eq: DwInvertible}
		z(s)
		=
		\int_s^t e^{-\mathcal{A}(\tau-s)} 
		\begin{bmatrix}
			0 \\
			3 \mathcal{S}(v,\xi)_1^2 z_1
		\end{bmatrix}
		\, \mathrm{d} \tau
		-f(s).
	\end{equation}
	We now show that for any given $f \in \mathcal{C}_t$ there exists exactly one $z \in \mathcal{C}_t$ satisfying \eqref{eq: DwInvertible}. To that end we adjust the technique applied in \cite[Thm.~4.3]{KunM20} and choose a partition $0 = t_0 < t_1 < \cdots < t_n = t$ such that 
	\begin{equation*}
		3 \, (t_{i+1} - t_i) \, c_{\mathrm{em}}^3 \, ~ \Vert \mathcal{S}^t(v,\xi) \Vert_{\mathcal{C}_t}^2 < \tfrac{1}{2}~~~\text{for all} ~ i = 0, \ldots,n-1,
	\end{equation*}
	where $c_\mathrm{em}$ is the constant associated with $H^1_0(\Omega) \hookrightarrow L^6(\Omega)$. For $i = 0, \ldots, n-1$ let $\xi_i \in \mathcal{E}$ be fixed and define the operators 
	\begin{equation*}
		\mathcal{T}_i \colon C([t_i,t_{i+1}];\mathcal{E}) \rightarrow C([t_i,t_{i+1}];\mathcal{E})
	\end{equation*}
	via $\mathcal{T}_i \, h = z$, where for $s \in [t_i,t_{i+1}]$
	\begin{equation*}
		z(s) = e^{-\mathcal{A}(t_{i+1}-s)} \xi_i
		+ \int_s^{t_{i+1}} e^{-\mathcal{A}(s-\tau)} \begin{bmatrix} 0 \\ 3 \mathcal{S}^t(v,\xi)_1^2(\tau) \, h_1(\tau) \end{bmatrix} \, \mathrm{d} \tau - f(s).
	\end{equation*}
	For $h_1,h_2 \in C([t_i,t_{i+1}];\mathcal{E})$ it holds
	\begin{equation*}
		\Vert \mathcal{T}_i h_1 - \mathcal{T}_i h_2 \Vert_{C([t_i,t_{i+1}];\mathcal{E})}
		\leq 3 c_\mathrm{em}^3 (t_{i+1}-t_i) \Vert \mathcal{S}^t(v,\xi) \Vert_{C([t_i,t_{i+1}]; \mathcal{E})}^2
		\Vert h_1 - h_2 \Vert_{C([t_i,t_{i+1}]; \mathcal{E})},
	\end{equation*}
	showing that $\mathcal{T}_i$ is a contraction. Banach's fixed point theorem implies existence of a unique fixed point $z_i \in C([t_i,t_{i+1}];\mathcal{E})$ depending on $\xi_i$. Iteratively setting $\xi_{n-1} = 0$ and $\xi_{n-1-i} = z_{n-i}(t_{n-i})$ for $i = 0,\ldots,n-2$ yields functions $z_i \in C([t_i,t_{i+1}];\mathcal{E})$, $i = 0,\ldots,n-1$ that can be glued together to obtain a solution of \eqref{eq: DwInvertible}. The uniqueness of the fixed points implies uniqueness of the solution to \eqref{eq: DwInvertible}.
	
	The implicit function theorem now yields the asserted regularity of $\mathcal{S}^t$ and equations \eqref{eq: Sder1} and \eqref{eq: Sder2}. Taking appropriate derivatives of \eqref{eq: Sder1} and \eqref{eq: Sder2} yields formulas for the second order partial derivatives. Due to the regularity of $w$ unique solvability of the linear equations \eqref{eq: Sder1} - \eqref{eq: Sder3} follows by reversing time and considering \cite[II Prop.~3.4]{BenEtAl07}.
\end{proof}
In the following we show appropriate estimates for the solution $\mathcal{S}^t(v,\xi)$.  Our local analysis, however, requires an estimate for the solution's distance to the nominal trajectory. Such bounds are obtained by application of a fixed point argument and only hold for sufficiently small $\xi$ and $v$. More specifically we show that the unique solution of the equation satisfied by $\mathcal{S}^t(v,\xi) - \Tilde{w}$ lies in a certain neighborhood of zero.
First we analyze the linear part of said equation using the definitions and techniques presented in \cite[II Sec.~3.4]{BenEtAl07}. 
\begin{lemma}\label{lem: estLinEq}
	Let $\Tilde{w}$ be the nominal trajectory and let $t \in (0,T]$, $f \in \mathcal{L}_t$ and $\xi \in \mathcal{E}$. Then there exists exactly one solution $z \in L^2(0,t;\mathcal{E})$ to 
	\begin{equation}\label{eq: linEq}
		\begin{aligned}
			z (s) 
			= e^{-\mathcal{A}(t-s)} \xi
			- \int_s^t e^{-\mathcal{A}(\tau-s)} \begin{bmatrix} 0 \\ - \Tilde{w}_1^2(\tau) z_1(\tau) + f(\tau) \end{bmatrix} \, \mathrm{d} \tau~~~s \in [0,t]. 
		\end{aligned}
	\end{equation}
	It holds $z \in \mathcal{C}_t$ and there exists $c_w > 0$ depending on $T$ but independent of $f$, $\xi$ and $t$ such that
	\begin{equation*}
		\left\Vert z \right\Vert_{\mathcal{C}_t}
		\leq
		c_w \left( \Vert \xi \Vert_\mathcal{E} + \Vert f \Vert_{\mathcal{L}_t}  \right).
	\end{equation*}
\end{lemma}
\begin{proof}
	A time transformation unveils that $z$ being a solution to \eqref{eq: linEq} is equivalent to $s \mapsto \overleftarrow{z}(s) \coloneqq z(t-s)$ being a solution to
	\begin{equation*}
		\overleftarrow{z}(s) = e^{-\mathcal{A}s} \xi 
		- \int_{0}^s e^{-\mathcal{A}(s-\tau)} \left( - F(\tau) \overleftarrow{z}(\tau) + \begin{bmatrix} 0 \\ f(t-\tau) \end{bmatrix} \right) \, \mathrm{d}\tau, 
	\end{equation*}
	where for $s \in [0,T]$ the operator $F(s) \colon \mathcal{E} \rightarrow \mathcal{E}$ is defined via
	\begin{equation}\label{eq: DefF}
		F(s)\begin{bmatrix} w_1 \\ w_2 \end{bmatrix} = \begin{bmatrix} 0 \\ 3 \Tilde{w}_1^2(s) \, w_1 \end{bmatrix}.
	\end{equation}
	Since $\Tilde{w} \in \mathcal{C}_T$, we have $F(s) \in \mathcal{L}(\mathcal{E})$ for every $s \in [0,T]$ and further $s \mapsto F(s)x$ is continuous from $[0,T]$ to $\mathcal{E}$ for every $x \in \mathcal{E}$. The existence of a unique solution then follows from \cite[II Prop.~3.4]{BenEtAl07}.
	The regularity of $\Tilde{w}$ yields uniform boundedness of $\Vert  F(s) \Vert_{\mathcal{L}(\mathcal{E};\mathcal{E})}$ in $s\in [0,T]$. Standard estimates of \eqref{eq: linEq} together with Gronwall's inequality yield the estimate for the solution.
\end{proof}
With this tool at hand the fixed point argument can be realized.
\begin{proposition}\label{prop: SolStateEq}
	There exists $\delta_1 > 0$ independent of $t$ such that for any $\xi \in \mathcal{E}$ and $v \in \mathcal{L}_t$ satisfying
	$\Vert \xi \Vert_\mathcal{E} + \Vert v \Vert_{\mathcal{L}_t} < \delta_1$
	it holds
	\begin{equation*}
		\left\Vert \mathcal{S}^t(v,\xi) - \Tilde{w} \right\Vert_{\mathcal{C}_t} 
		\leq 2 c_w \left( \Vert \xi \Vert_\mathcal{E} + \Vert v \Vert_{\mathcal{L}_t} \right),
	\end{equation*}
	where $c_w$ is the constant from \Cref{lem: estLinEq}.
\end{proposition}
\begin{proof}
	We define the set
	\begin{equation*}
		\mathcal{M} = \left\{
		w \in \mathcal{C}_t \colon  
		\left\Vert w \right\Vert_{\mathcal{C}_t} 
		\leq 2 c_w \left( \Vert \xi \Vert_\mathcal{E}
		+ \Vert v \Vert_{\mathcal{L}_t} \right) \right\}
	\end{equation*}
	and the operator
	\begin{equation*}
		\mathcal{Z} \colon \mathcal{M} \rightarrow \mathcal{C}_t,~~~ w = \begin{bmatrix} w_1 \\ w_2 \end{bmatrix} \mapsto z,
	\end{equation*}
	where $z$ is the unique solution of \eqref{eq: linEq} with $f = -w_1^3 - 3 \Tilde{w}_1 w_1^2 + v$. Note that by \Cref{lem: cubicInL2} it holds $f \in \mathcal{L}_t$. Now any fixed point $\bar{w} = \begin{bmatrix} \bar{w}_1 \\ \bar{w}_2 \end{bmatrix} \in \mathcal{M}$ of $\mathcal{Z}$ satisfies for $s\in [0,t]$
	\begin{equation*}
		\begin{bmatrix} \bar{w}_1 \\ \bar{w}_2 \end{bmatrix} (s)
		= e^{-\mathcal{A}(t-s)} \xi 
		- \int_s^t e^{-\mathcal{A}(\tau-s)} 
		\begin{bmatrix} 0 \\ - \bar{w}_1^3 - 3 \Tilde{w}_1 \bar{w}_1^2 - 3 \Tilde{w}_1^2 \bar{w}_1 + v \end{bmatrix} \,\mathrm{d}\tau.
	\end{equation*}
	The analogue of \eqref{eq: mildSoluBack} for $\Tilde{w}$ implies
	\begin{equation*}
		\Tilde{w}(s) + \bar{w} (s)
		=
		e^{-\mathcal{A}(t-s)} \left( \Tilde{w}(t) + \xi \right)
		-
		\int_s^t e^{-\mathcal{A}(\tau-s)} \begin{bmatrix} 0 \\ -(\Tilde{w}_1 - \bar{w}_1)^3 + v \end{bmatrix} \, \mathrm{d}\tau
	\end{equation*}
	and hence $\Tilde{w} + \bar{w} = \mathcal{S}^t(v,\xi)$.
	The estimate holds due to the fact that $\mathcal{S}^t(v,\xi) - \Tilde{w} = \bar{w} \in \mathcal{M}$.
	
	It remains to show that $\mathcal{Z}$ is a contraction on $\mathcal{M}$. Then Banach's fixed point theorem yields existence of a fixed point. 
	We first show $\mathcal{Z}(\mathcal{M}) \subset \mathcal{M}$. To that end note that for any $w = \begin{bmatrix} w_1 \\ w_2 \end{bmatrix} \in \mathcal{M}$ it holds
	\begin{equation*}
		\Vert w \Vert_{\mathcal{C}_t} \leq 2 c_w \left( \Vert \xi \Vert_\mathcal{E}
		+ \Vert v \Vert_{\mathcal{L}_t} \right)
		< 2 c_w \delta_1. 
	\end{equation*}
	With \Cref{lem: cubicInL2} it follows
	\begin{equation*}
		\begin{aligned}
			\Vert w_1^3 + 3\Tilde{w}_1 w_1^2 \Vert_{\mathcal{L}_t}
			&\leq \sqrt{2 T c_\mathrm{em}^3} \left( \Vert w \Vert_{\mathcal{C}_t}^2 + 3 \left\Vert \Tilde{w} \right\Vert_{\mathcal{C}_t} \Vert w \Vert_{\mathcal{C}_t} \right)
			\Vert w \Vert_{\mathcal{C}_t}\\
			&\leq \sqrt{2 T c_\mathrm{em}^3} \left( 4 c_w^2 \delta_1^2 + 6 c_w \Vert \Tilde{w} \Vert_{\mathcal{C}_t} \delta_1 \right) 2 c_w \left( \Vert \xi \Vert_\mathcal{E}
			+ \Vert v \Vert_{\mathcal{L}_t} \right)
		\end{aligned}
	\end{equation*}
	and for a sufficiently small $\delta_1$ one obtains $\sqrt{2 T c_\mathrm{em}^3} \left( 4 c_w^2 \delta_1^2 + 6 c_w \Vert \Tilde{w} \Vert_{\mathcal{C}_t} \delta_1 \right) 2 c_w < 1$. Utilizing \Cref{lem: estLinEq} it follows
	\begin{equation*}
		\begin{aligned}
			\Vert \mathcal{Z}(w) \Vert_{\mathcal{C}_t}
			&\leq c_w \left( \Vert \xi \Vert_\mathcal{E}
			+ \Vert -w_1^3 - 3 \Tilde{w}_1 w_1^2 + v \Vert_{\mathcal{L}_t} \right)\\
			&\leq c_w \left( \Vert \xi \Vert_\mathcal{E} 
			+ \Vert v \Vert_{\mathcal{L}_t}
			+ \Vert w_1^3 + 3 \Tilde{w}_1 w_1^2 \Vert_{\mathcal{L}_t} 
			\right)
			\leq 2 c_w \left( \Vert \xi \Vert_\mathcal{E} 
			+ \Vert v \Vert_{\mathcal{L}_t} \right),
		\end{aligned}
	\end{equation*}
	implying $\mathcal{Z}(w) \in \mathcal{M}$.
	
	To show the Lipschitz continuity of $\mathcal{Z}$ assume $w$,$q \in \mathcal{M}$. Then $z = \mathcal{Z}(q) - \mathcal{Z}(w)$ satisfies \eqref{eq: linEq} with $\xi = 0$ and $f = -(q_1^3 - w_1^3) - 3 \Tilde{w}_1 (q_1^2 - w_1^2)$. Hence
	\begin{equation}\label{eq: FPL1}
		\Vert \mathcal{Z}(q) - \mathcal{Z}(w) \Vert_{\mathcal{C}_t}
		\leq c_w \left( \Vert q_1^3 - w_1^3 \Vert_{\mathcal{L}_t} + 3 \Vert \Tilde{w}_1 (q_1^2 - w_1^2) \Vert_{\mathcal{L}_t} \right). 	
	\end{equation}
	Using the identities
	\begin{equation}\label{eq: cubId}
		\begin{aligned}
			q_1^3 - w_1^3 &= (q_1 - w_1)^3 + 3 w_1 (q_1 - w_1)^2 + 3 q_1^2 (w_1 - q_1),\\
			q_1 - w_1^2 &= (q_1 + w_1) (q-1 - w_1), 
		\end{aligned}
	\end{equation}
	and the estimate implied by $q_1,w_1 \in \mathcal{M}$ it is straightforward to show that for a sufficiently small $\delta_1$ the operator $\mathcal{Z}$ is a contraction.
\end{proof}

We are now in a position to estimate the operator norms of the derivatives of $\mathcal{S}^t$ independently of $t$. The following lemma will be needed for the sensitivity analysis of the value function.
\begin{lemma}\label{lem: EstS}
	Let $t \in (0,T]$ be fixed and $\Vert (\xi,v) \Vert_{\mathcal{E} \times \mathcal{L}_t} < \tfrac{1}{2} \delta_1.$
	For $i = 1,...,5$ let $u_i \in \mathcal{L}_t$ and $\eta_i \in \mathcal{E}$.
	Then there exists $c_{\mathcal{S}} > 0$ independent of $t$ such that
	\begin{equation*}
		\Vert \mathcal{S}^t(v,\xi) \Vert_{\mathcal{C}_t}
		\leq c_{\mathcal{S}}
	\end{equation*}
	and for $i,j \in \{ 0,1,...,5 \}$ such that $i+j \leq 5$ it holds
	\begin{equation*}
		\Vert D_{v^i \xi^j}^{i+j} \mathcal{S}^t(v,\xi) [\eta_1,...,\eta_j,u_1,...,u_i] \Vert_{\mathcal{C}_t}
		\leq c_{\mathcal{S}} \left( \prod_{k=1}^i \Vert u_k \Vert_{\mathcal{L}_t} \right)
		\left( \prod_{k=1}^j \Vert \eta_k \Vert_{\mathcal{E}} \right).
	\end{equation*}
\end{lemma}
\begin{proof}
	The first estimate follows directly from \Cref{prop: SolStateEq}. We only present the proof for $D_v \mathcal{S}^t$. The remaining estimates can be shown analogously in an iterative manner by noting that for derivatives of order higher than two characterizations analogous to the ones in \Cref{lem: SmoothS} can be derived. 
	
	Again denote $w = \mathcal{S}^t(v,\xi)$ and for some fixed $u \in \mathcal{L}_t $ we denote $\zeta = (\zeta_1, \zeta_2) = D_v \mathcal{S}^t(v,\xi)[u] \in \mathcal{C}_t$. According to \Cref{lem: SmoothS} for $s \in [0,t]$ it holds
	\begin{equation*}
		\zeta(s) = 
		- \int_s^t e^{-\mathcal{A}(\tau-s)} \begin{bmatrix} 0 \\ -3 w_1^2 \zeta_1 + u \end{bmatrix} \, \mathrm{d} \tau.
	\end{equation*}
	Since $t \mapsto e^{-\mathcal{A}t} $ is a contraction semigroup, for any $s \in [0,t)$ it follows
	\begin{equation*}
	\begin{aligned}
		\Vert \zeta (s) \Vert_{\mathcal{E}}
		&\leq 
		\int_s^t 3 \Vert w_1^2(\tau) \zeta_1(\tau) \Vert_{L^2(\Omega)} 
		+  \Vert u (\tau) \Vert_{L^2(\Omega)}\, \mathrm{d} \tau\\
		&\leq 
		\sqrt{T} \Vert u \Vert_{\mathcal{L}_t}
		+ 3c_\mathrm{em}^3 \int_s^t \Vert w (\tau) \Vert_\mathcal{E}^2 \Vert \zeta(\tau) \Vert_\mathcal{E} \, \mathrm{d} \tau,
	\end{aligned}
	\end{equation*}
	where the second estimate is justified by \Cref{lem: cubicInL2}. Gronwall's inequality yields
	\begin{equation*}
		\Vert \zeta (s) \Vert_{\mathcal{E}}
		\leq \sqrt{T} \Vert u \Vert_{\mathcal{L}_t} 
		\exp\left( \int_s^t 3c_\mathrm{em}^3 \Vert w \Vert_\mathcal{E}^2 \, \mathrm{d} \tau \right).
	\end{equation*}
	Since $s$ was arbitrary,
	the asserted estimate follows with the estimate for $\Vert w \Vert_{\mathcal{C}_t} = \Vert \mathcal{S}^t(v,\xi) \Vert_{\mathcal{C}_t}$.
\end{proof}
%
\subsection{Existence of a minimizer}
%
In this subsection existence of a minimizing pair is shown. For any $\epsilon > 0$ and $t \in (0,T]$ we define the corresponding set of data pairs via 
\begin{equation*}
	\mathcal{D}_\epsilon^t \coloneqq \{ (\xi,y) \in \mathcal{E} \times \mathcal{Y}_t ~\colon~ \Vert (\xi,y) \Vert_{\mathcal{E} \times \mathcal{Y}_t} < \epsilon \}.
\end{equation*}
\begin{proposition}\label{prop: exMin}
	There exists $\delta_2 \in \left(0,\tfrac{1}{2}\delta_1\right]$ independent of $t$ such that for any $(\xi,y) \in \mathcal{D}_{\delta_2}^t$ the optimal control problem \eqref{eq: OCP} admits a minimizer $(\bar{w},\bar{v}) \in \mathcal{C}_t \times \mathcal{L}_t$. Further there exists $c_\mathrm{Min} > 0$ independent of $t$, $\xi$ and $y$ such that any minimizing pair $(z,u)$ satisfies
	\begin{equation*}
		\max \left(\Vert z - \Tilde{w} \Vert_{\mathcal{C}_t}, \Vert u \Vert_{\mathcal{L}_t} \right) 
		\leq c_\mathrm{Min} \max \left( \Vert \xi \Vert_{\mathcal{E}}, \Vert y \Vert_{\mathcal{Y}_t} \right)
	\end{equation*}
	and in particular it holds $\Vert u \Vert_{\mathcal{L}_t} < \tfrac{1}{2}\delta_1$.
\end{proposition}
\begin{proof}
	For now set $\delta_2 = \tfrac{1}{2} \delta_1$ and let $\Vert \xi \Vert_\mathcal{E} < \delta_2$. For any $v \in \mathcal{L}_t$ we denote $w_v = \mathcal{S}^t(v,\xi)$. We further set
	$j = \inf\limits_{v \in \mathcal{L}_t} J(w_v,v;t,y).$
	\Cref{prop: SolStateEq} ensures that for $w^0 = \mathcal{S}^t(0,\xi) \in \mathcal{C}_t$ it holds
	\begin{equation*}
		\begin{aligned}
			j 
			&\leq J(w^0,0;t,y) 
			= \frac{1}{2} \Vert w^0(0) - w_0 \Vert_{\mathcal{E}}^2
			+ \frac{\alpha}{2} \left\Vert y - C \left( w^0 - \Tilde{w} \right) \right\Vert_{\mathcal{Y}_t}^2 \\
			&\leq \frac{1}{2} \Vert w^0 - \Tilde{w} \Vert_{\mathcal{C}_t}^2
			+ \alpha \Vert y \Vert_{\mathcal{Y}_t}^2
			+ \alpha T c_C^2 \Vert w^0 - \Tilde{w} \Vert_{\mathcal{C}_t}^2
			\leq c \Vert \xi \Vert_\mathcal{E}^2 + \alpha \Vert y \Vert_{\mathcal{Y}_t}^2
			\leq c \delta_2^2,			
		\end{aligned}
	\end{equation*}
	where again $c_C = \Vert C \Vert_{\mathcal{L}(\mathcal{E};Y)}$.
	Hence $0 \leq j \leq c \delta_2^2 < \infty$ and there exists a sequence $(w_k,v_k)$ such that $w_k = \mathcal{S}^t(v_k,\xi) $ and $J(w_k,v_k;t,y) \to j$, for $k \to \infty$. If $j > 0$ it can be assumed that $J(w_k,v_k;t,y) \leq 2j$ and it follows
	\begin{equation}\label{eq: ExPr1}
		J(w_k,v_k,t,y) \leq c \delta_2^2.
	\end{equation}
	If $j = 0$, \eqref{eq: ExPr1} can be assumed directly. Hence
	\begin{equation*}
		\Vert v_k \Vert_{\mathcal{L}_t}^2 
		\leq 2 J(w_k,v_k;t,y) 
		\leq c \delta_2^2
	\end{equation*}
	implying that $(v_k)$ is bounded in $\mathcal{L}_t$. It therefore admits a subsequence (also denoted by $(v_k)$) such that
	\begin{equation}\label{eq: weakConvV}
		v_k \rightharpoonup \bar{v}~~~\text{ in } \mathcal{L}_t. 	
	\end{equation}
	Turning to $(w_k)$ we find that after an appropriate decrease of $\delta_2$ it holds 
	\begin{equation*}
		\Vert v_k \Vert_{\mathcal{L}_t} 
		\leq c \delta_2 < \tfrac{1}{2} \delta_1
	\end{equation*}
	and \Cref{prop: SolStateEq} can be applied. We obtain
	\begin{equation*}
		\Vert w_k \Vert_{\mathcal{C}_t} 
		\leq \Vert w_k - \Tilde{w} \Vert_{\mathcal{C}_t} 
		+ \Vert \Tilde{w} \Vert_{\mathcal{C}_t}
		\leq 2 c_w \left( \Vert \xi \Vert_{\mathcal{E}} 
		+ \Vert v_k \Vert_{\mathcal{L}_t} \right) 
		+ \Vert \Tilde{w} \Vert_{\mathcal{C}_t}. 
	\end{equation*}
	Hence $(w_k)$ is bounded in $\mathcal{C}_t$ implying that it is bounded in $L^\infty(0,t;\mathcal{E})$. We obtain a subsequence $(w_k)$ such that
	\begin{equation}\label{eq: weakConvL2H1}
		w_k \overset{\ast}{\rightharpoonup} \bar{w}~~~ \text{ in } L^\infty(0,t;\mathcal{E}) 
	\end{equation}
	for some element $\bar{w} \in L^\infty(0,t;\mathcal{E})$.
	To show that $\bar{w}$ is the solution associated with $\bar{v}$ we switch to the setting of weak solutions. We first note that \eqref{eq: weakConvL2H1} implies $\text{weak}^*$ convergence of $w_{k,1}$ and $\partial_t w_{k,1} = w_{k,2}$ in $L^\infty(0,t;L^2(\Omega))$. Hence it follows $\bar{w}_2 = \partial_t \bar{w}_1$.
	In accordance with \Cref{cor: BackwWP} we denote 
	\begin{equation*}
		\begin{bmatrix} u_k \\ \partial_t u_k \end{bmatrix}
		= 
		w_k
		~~~~~
		\text{and}
		~~~~~
		\begin{bmatrix} \bar{u} \\ \partial_t \bar{u} \end{bmatrix}
		=
		\bar{w}.
	\end{equation*}
	We get that $u_k$ is the unique weak solution associated with $v_k$ and \eqref{eq: weakConvL2H1} implies 
	\begin{equation}\label{eq: weakConvL2H1_u}
		\begin{alignedat}{2}
			u_k &\overset{\ast}{\rightharpoonup} \bar{u} ~~~ &&\text{ in } L^\infty(0,t;H^1_0(\Omega)),\\ 
			\partial_t u_k &\overset{\ast}{\rightharpoonup} \partial_t \bar{u} ~~~ &&\text{ in } L^\infty(0,t;L^2(\Omega)). 
		\end{alignedat}
	\end{equation}
	In the following we pass to the limit in the weak formulation to show that $\bar{u}$ is the weak solution associated with $\bar{v}$ which then implies that $\bar{w}$ is the mild solution. 
	For all $k$ and $\varphi \in C_c^\infty((0,t)\times \Omega)$ it holds
	\begin{equation}\label{eq: passLim}
		\begin{aligned}
			&- \int_0^t (\partial_t u_k(s), \partial_t \varphi(s))_{L^2(\Omega)} \,\mathrm{d}s
			+ \int_0^t (u_k^3(s), \varphi(s))_{L^2(\Omega)} \,\mathrm{d}s\\
			&= \int_0^t (v_k(s),\varphi(s))_{L^2(\Omega)} \,\mathrm{d}s
			- \int_0^t (\nabla u_k(s), \nabla \varphi(s))_{L^2(\Omega)} \,\mathrm{d}s.
		\end{aligned}
	\end{equation}
	Convergence of the linear terms is guaranteed by \eqref{eq: weakConvV} and \eqref{eq: weakConvL2H1_u}.
	We turn to the nonlinear term and note that $H^1_0(\Omega) \doublehookrightarrow L^3(\Omega) \hookrightarrow L^2(\Omega)$. Due to \eqref{eq: weakConvL2H1_u} $(u_k)$ and $(\partial_t u_k)$ are bounded in $L^\infty(0,t;H^1_0(\Omega))$ and in $L^\infty(0,t;L^2(\Omega))$, respectively. An application of \cite[Cor.~4]{Sim86} yields existence of a subsequence $(u_k)$ satisfying
	\begin{equation}\label{eq: convInCLp}
		u_k \rightarrow \bar{u}~~~ \text{ in } C([0,t];L^3(\Omega)).
	\end{equation}
	It follows that $u_k^3$ converges to $\bar{u}^3$ in $C([0,t];L^1(\Omega))$.
	Passing to the limit in \eqref{eq: passLim} verifies that for any $\varphi \in C_c^\infty((0,t) \times \Omega)$ it holds
	\begin{equation*}
		\begin{aligned}
			&- \int_0^t (\partial_t \bar{u}(s), \partial_t \varphi(s))_{L^2(\Omega)} \,\mathrm{d}s
			+ \int_0^t (\bar{u}^3(s), \varphi(s))_{L^2(\Omega)} \,\mathrm{d}s\\
			&= \int_0^t (\bar{v}(s),\varphi(s))_{L^2(\Omega)} \,\mathrm{d}s
			- \int_0^t (\nabla \bar{u}(s), \nabla \varphi(s))_{L^2(\Omega)} \,\mathrm{d}s.
		\end{aligned}
	\end{equation*}
	It remains to show that the final condition is fulfilled. Since $u_k$ is the weak solution with respect to $v_k$, for any $\varphi \in C^\infty_c((0,t) \times \Omega)$ and almost all $s \in (0,t)$ it holds
	\begin{equation*}
		\begin{aligned}
			&\left\vert \langle \partial_{tt} u_k(s),\varphi(s) \rangle_{H^{-1}(\Omega) \times H^1_0(\Omega)} \right\vert\\
			&\leq \left\vert (u_k^3(s),\varphi(s))_{L^2(\Omega)} \right\vert 
			+ \left\vert (v_k(s),\varphi(s))_{L^2(\Omega)} \right\vert
			+ \left\vert (\nabla u_k(s),\nabla \varphi(s))_{L^2(\Omega)} \right\vert\\
			&\leq c \left( \Vert u_k(s) \Vert_{H^1_0(\Omega)}^3 
			+ \Vert v_k(s) \Vert_{L^2(\Omega)}
			+ \Vert u_k(s) \Vert_{H^1_0(\Omega)} \right) \Vert \varphi \Vert_{H^1_0(\Omega)}.
		\end{aligned}
	\end{equation*}
	It follows
	\begin{equation*}
		\begin{aligned}
			\int_0^t \Vert \partial_{tt} u_k(s) \Vert_{H^{-1}(\Omega)}^2 \, \mathrm{d}s
			\leq c \int_0^t \Vert u_k(s) \Vert_{H^1_0(\Omega)}^6 
			+ \Vert v_k(s) \Vert_{L^2(\Omega)}^2
			+ \Vert u_k(s) \Vert_{H^1_0(\Omega)}^2 \, \mathrm{d}s.
		\end{aligned}
	\end{equation*}
	Since $(u_k)$ and $(v_k)$ are bounded in $L^\infty(0,t;H^1_0(\Omega))$ and in $\mathcal{L}_t$, respectively, $(\partial_{tt} u_k)$ is bounded in $L^2(0,t;H^{-1}(\Omega))$. The embedding $L^2(\Omega) \doublehookrightarrow H^{-1}(\Omega)$ allows another application of \cite[Cor.~4]{Sim86} to obtain a subsequence such that
	\begin{equation}\label{eq: convInCH-1}
		\partial_t u_k \rightarrow \partial_t \bar{u}~~~ \text{ in } C([0,t];H^{-1}(\Omega)). 
	\end{equation}
	Now \eqref{eq: convInCLp} and \eqref{eq: convInCH-1} allow us to conclude
	\begin{equation*}
		\begin{aligned}
			\Tilde{w}_1(t) + \xi_1 &= u_k(t) \rightarrow \bar{u}(t)~\text{ in } L^3(\Omega)\\
			\partial_t \Tilde{w}_1(t) + \xi_2 &= \partial_t u_k(t), \rightarrow \partial_t \bar{u}(t)~\text{ in } H^{-1}(\Omega),
		\end{aligned}
	\end{equation*}
	and the final condition is satisfied. Therefore $\bar{u}$ is the weak solution associated with $\bar{v}$ and we conclude that $\bar{w}$ is the mild solution, i.e., $\bar{w} = w_{\bar{v}} \in \mathcal{C}_t$.
	
	In order to show that $(\bar{w},\bar{v})$ is a minimizing pair we require stronger convergence of $(w_k(0))$ than what follows from \eqref{eq: convInCLp} and \eqref{eq: convInCH-1}. To that end note that the boundedness of $(w_k)$ in $\mathcal{C}_t$ implies boundedness of $(w_k(0))$ in $\mathcal{E}$ and we extract a subsequence satisfying
	$w_k(0) \rightharpoonup \bar{w}(0)$ in $\mathcal{E}$.
	With the weak lower semi continuity of norms it follows that $(\bar{w},\bar{v})$ is indeed a minimizer.
	
	For any minimizer $(z,u)$ it holds $J(z,u;t,y) = j$, hence the estimate can be shown using the techniques deployed above together with \Cref{prop: SolStateEq}. 
\end{proof}
%
\subsection{Optimality conditions}
In this subsection we introduce the \textit{adjoint state} and present a first order necessary optimality condition. Note that due to a lack of regularity the adjoint is set via an ad hoc definition rather than derived as the solution of an adjoint equation as was done, e.g., in \cite{KunM20}.
\begin{definition}
	For any fixed $t \in (0,T]$ and $(\xi,\bar{v}) \in \mathcal{E} \times \mathcal{L}_t$ satisfying $\max \left(\Vert \xi \Vert_\mathcal{E}, \Vert \bar{v} \Vert_{\mathcal{L}_t} \right) < \frac{\delta_1}{2}$ denote $\bar{w} = \mathcal{S}^t(\bar{v},\xi)$. We define the associated adjoint state $p \in \left(\mathcal{L}_t \right)^* = \mathcal{L}_t$ as
	\begin{equation*}
		p = D_v \mathcal{S}^t(\bar{v},\xi)^* \left(
		\delta_0^* (\bar{w} (0) - w_0)
		- \alpha \Psi[ \bar{w} - \Tilde{w},y]
		\right), 
	\end{equation*}
	where for $z \in \mathcal{C}_t$ the bounded linear point evaluation operator $\delta_0 \colon \mathcal{C}_t \rightarrow \mathcal{E}$ and its adjoint $\delta_0^* \colon \mathcal{E} \rightarrow \mathcal{C}_t^*$ are defined as
	\begin{equation*}
		\delta_0 (z) = z(0),
		~~~~~
		\langle \delta_0^* (\xi) , z \rangle_{\mathcal{C}_t^*,\mathcal{C}_t} = ( \xi, z(0) )_\mathcal{E}.
	\end{equation*}
	For $z \in \mathcal{C}_t$ and $y \in \mathcal{Y}_t$ we define $\Psi[z,y] \in \mathcal{C}_t^*$ as
	\begin{equation*}
		\langle \Psi[z,y], \zeta \rangle_{\mathcal{C}_t^*,\mathcal{C}_t}
		= \int_0^t \left\langle C^* \left(y - C z \right), \zeta \right\rangle_{\mathcal{E}^*,\mathcal{E}} \, \mathrm{d} s,
		~~~~~ \zeta \in \mathcal{C}_t.
	\end{equation*}
\end{definition}
Before we state the first order optimality condition the adjoint state is estimated from above.
\begin{lemma}\label{lem: estAdj}
	Let $(\xi,y) \in \mathcal{D}_{\delta_2}^t$. Let $(\bar{w},\bar{v})$ be a minimizer of \eqref{eq: OCP} with associated adjoint $p$. There exists a constant $c_p > 0$ independent of $t$, $\xi$ and $y$ such that
	\begin{equation*}
		\Vert p \Vert_{\mathcal{L}_t}
		\leq c_p \max\left( \Vert \xi \Vert_\mathcal{E}, \Vert y \Vert_{\mathcal{Y}_t} \right).
	\end{equation*}
\end{lemma}
\begin{proof}
	First consider
	\begin{equation}\label{eq: estPsi1}
		\sup_{\Vert \varphi \Vert_{\mathcal{C}_t} = 1 }
		\vert \langle \Psi[z,y],\varphi \rangle_{\mathcal{C}_t^*,\mathcal{C}_t} \vert
		\leq \int_0^t \left\Vert C^* \left(y - C z \right) \right\Vert_{\mathcal{E}^*} \, \mathrm{d} s 
		\leq \sqrt{T} c_C \Vert y \Vert_{\mathcal{Y}_t}		
		+  T c_C^2  \Vert z \Vert_{\mathcal{C}_t} .
	\end{equation}
	Utilizing \Cref{prop: exMin} and defining an appropriate $c_p > 0$ it follows
	\begin{equation}\label{eq: estPsi}
		\Vert \Psi[\bar{w} - \Tilde{w},y] \Vert_{\mathcal{C}_t^*}
		\leq c_p \max \left( \Vert \xi \Vert_{\mathcal{E}}, \Vert y \Vert_{\mathcal{Y}_t} \right).
	\end{equation}
	Due to the definition of $p$ it holds
	\begin{equation*}
		\begin{aligned}
			\Vert p \Vert_{\mathcal{L}_t}
			\leq \Vert D_v\mathcal{S}^t(\bar{v},\xi)^* \Vert_{\mathcal{L}(\mathcal{C}_t^*;\mathcal{L}_t ) } 
			\left( \Vert \delta_0^* \Vert_{\mathcal{L}(\mathcal{E}^*;\mathcal{C}_t^* )} \Vert \bar{w} - \Tilde{w} \Vert_{\mathcal{C}_t} 
			+ \alpha \Vert \Psi[\bar{w} - \Tilde{w},y] \Vert_{\mathcal{C}_t^*} 
			\right).
		\end{aligned}
	\end{equation*}
	Using \eqref{eq: estPsi}, \Cref{prop: exMin}, \Cref{lem: EstS}, the identity $ \Vert \delta_0^* \Vert_{\mathcal{L}(\mathcal{C}_t;\mathcal{E})} = 1$ one can adjust $c_p$ such that the assertion holds.
\end{proof}
The first order necessary condition follows from the definition of $p$.
\begin{proposition}\label{prop: optCond}
	For any $(\xi,y) \in \mathcal{D}_{\delta_2}^t$ and any minimizing pair $(\bar{w},\bar{v})$ of the corresponding optimal control problem \eqref{eq: OCP} it holds
	\begin{equation*}
		\bar{v} = - p ~~~~~ \text{ in } \mathcal{L}_t,
	\end{equation*}
	where $p$ is the adjoint state associated with $(\bar{w},\bar{v})$.
\end{proposition}
\begin{proof}
	Let $(\bar{w},\bar{v})$ be a minimizing pair. Then the estimate given in \Cref{prop: exMin} ensures that $\bar{v}$ is a minimizer of the reduced cost functional
	\begin{equation*}
		J_\mathrm{r} \colon \left\{ v ~\colon~ \Vert v \Vert_{\mathcal{L}_t}  < \tfrac{1}{2} \delta_1 \right\}
		\rightarrow \mathbb{R},~~~
		J_\mathrm{r} (v) = J(\mathcal{S}(v,\xi),v;t,y).
	\end{equation*}
	The first order optimality condition $D J_\mathrm{r} (\bar{v}) = 0$ together with the chain rule yields the assertion.	
\end{proof}
We conclude the section by highlighting one of the major differences in the analysis of the finite-dimensional and the infinite-dimensional case.
\begin{remark}\label{rem: OptTrajNotClassical}
	A well-known implication of the optimality condition is that optimal controls may inherit higher regularity from the adjoint state. In the finite-dimensional case discussed in \cite{BreS24} this has tremendous consequences, namely that optimal trajectories are \textit{classical solutions} of the state equation. This ultimately leads to $C^1$-regularity of the value function, implying that it is a \textit{classical solution} of the associated Hamilton-Jacobi-Bellman equation. In contrast, in the context analyzed in this work no such effect occurs. 
\end{remark}
%

%
\section{Regularity of the value function}\label{sec: RegOfValFun}
%
This section discusses the regularity of the value function with respect to the spatial variable and the measurement. 
For $t \in (0,T]$ and $(\xi,y) \in \mathcal{E} \times \mathcal{Y}_t$ it is defined as
\begin{equation*}
	\mathcal{V}(t,\xi,y) = \inf\limits_{(w,v) \in \mathcal{C}_t \times \mathcal{L}_t} 
	J(w,v;t,y),~ \text{ subject to } e^t(w,v,\xi) = 0.
\end{equation*}
In time zero we set $\mathcal{V}(0,\xi,y) = \frac{1}{2} \Vert \xi  \Vert_{\mathcal{E}}^2$.
In order to show smoothness we apply the inverse mapping theorem to a mapping representing the optimality condition. Define the two spaces
\begin{equation*}
	X_t = \mathcal{C}_t \times \mathcal{L}_t \times \mathcal{L}_t \times \mathcal{Y}_t \times \mathcal{E},
	~~~~~~~
	Y_t = \mathcal{E} \times \mathcal{Y}_t \times \mathcal{C}_t \times \mathcal{L}_t \times \mathcal{L}_t
\end{equation*}
and the mapping $ \Phi_t \colon X_t \rightarrow Y_t $ as 
\begin{equation*}
	\Phi_t(\bar{w},\bar{v},p,y,\xi)
	=
	\begin{bmatrix}
		\xi \\
		y \\
		S^t(\bar{v},\xi) - \bar{w} \\
		D_v \mathcal{S}^t(\bar{v},\xi)^* \left(
		\delta_0^* \left( \bar{w}(0) - w_0 \right)
		- \alpha \Psi[\bar{w} - \Tilde{w},y] 
		\right) - p \\
		\bar{v} + p
	\end{bmatrix}.
\end{equation*}
Note that it holds $\Phi_t(w,v,p,y,\xi) = (\xi,y,0,0,0)$ if and only $(w,v,p)$ satisfies the first order optimality condition associated with the control problem given by $\xi$ and $y$.
Before applying the inverse mapping theorem some technical preparations are in order.
\begin{lemma}\label{lem: estPhi}
	Denote $h_0 = (\Tilde{w},0,0,0,0) \in X_t$ and let $h = (\bar{w},\bar{v},p,y,\xi) \in X_t$ be such that $\Vert h - h_0 \Vert_{X_t} < \delta_2$.
	Then the mapping $D\Phi_t(h_0) \in \mathcal{L}(X_t;Y_t)$ is bijective and there exists  $c_{\Phi} > 0$ independent of $t$ such that it holds
	\begin{equation*}
		\begin{aligned}
			\Vert D\Phi_t(h_0) - D\Phi_t(h)
			\Vert_{\mathcal{L}(X_t;Y_t)}
			&\leq c_{\Phi} \Vert h_0 - h \Vert_{X_t},\\
			\Vert D\Phi_t(h_0)^{-1} \Vert_{\mathcal{L}(Y_t;X_t)}
			&\leq c_{\Phi}.
		\end{aligned}
	\end{equation*}
\end{lemma}
\begin{proof}
	In order to show the first estimate let $\Vert (z,u,q,\gamma,\eta) \Vert_{X_t} = 1$. Then
	\begin{equation*}
		\Vert D\Phi_t(h_0)[(z,u,q,\gamma,\eta)] - D\Phi_t(h)[(z,u,q,\gamma,\eta)] \Vert_{Y_t}
	\end{equation*}
	is given by the maximum of the two terms
	\begin{equation}\label{eq: PhiProof1}
		\Vert D\mathcal{S}^t(0,0)[(u,\eta)] - D\mathcal{S}^t(\bar{v},\xi)[(u,\eta)] \Vert_{\mathcal{C}_t}
	\end{equation}
	and
	\begin{equation}\label{eq: PhiProof2}
		\begin{aligned}
			& \Vert D_v \mathcal{S}^t(0,0)^* \left( \delta_0^* z(0) - \alpha \Psi[z,\gamma] \right)
			- D_v \mathcal{S}^t(\bar{v},\xi)^* \left( \delta_0^* z(0) - \alpha \Psi[z,\gamma] \right) \\
			&- D(D_v\mathcal{S}^t(\bar{v},\xi)^*)[(u,\eta)] \left( \delta_0^* (\bar{w}(0) - w_0) - \alpha \Psi[\bar{w} - \zeta_{\Tilde{w}},y] \right) \Vert_{\mathcal{L}_t}, 
		\end{aligned}
	\end{equation}
	which will be estimated separately. Of course \eqref{eq: PhiProof2} can be estimated from above by 
	\begin{align}
		\tag*{A}\label{termA}
		& \Vert \left( D_v \mathcal{S}^t(0,0)^* - D_v \mathcal{S}^t(\bar{v},\xi)^* \right) \left( \delta_0^* z(0) - \alpha \Psi[z,\gamma] \right) \Vert_{\mathcal{L}_t}\\ 
		\tag*{B}\label{termB} 
		+& \Vert D(D_v\mathcal{S}^t(\bar{v},\xi)^*)[(u,\eta)] \left( \delta_0^*( \bar{w}(0) - w_0) - \alpha \Psi[\bar{w} - \zeta_{\Tilde{w}},y] \right) \Vert_{\mathcal{L}_t}.		
	\end{align}
	Using \eqref{eq: estPsi1} the term \ref{termA} is estimated from above by
	\begin{equation*}
		\Vert D_v \mathcal{S}^t(0,0) - D_v \mathcal{S}^t(\bar{v},\xi) \Vert_{\mathcal{L}(\mathcal{L}_t;\mathcal{C}_t)}
		\underbrace{ \left( \Vert z \Vert_{\mathcal{C}_t} + c (\Vert \gamma \Vert_{\mathcal{Y}_t} + \Vert z \Vert_{\mathcal{C}_t}) \right)}_{\leq \text{constant}},
	\end{equation*}
	where $c$ is an appropriately chosen constant independent of $t$. Using Taylor's theorem the operator norm on the left hand side is estimated from above by 
	\begin{equation*}
		\begin{aligned}
			&\sup_{\tau \in [0,1]} \Vert DD_v \mathcal{S}^t(\tau\bar{v},\tau\xi)[(\bar{v},\xi)] \Vert_{\mathcal{L}(\mathcal{L}_t;\mathcal{C}_t)}\\
			&\leq\sup_{\tau \in [0,1]} \Vert DD_v \mathcal{S}^t(\tau\bar{v},\tau\xi) \Vert_{\mathcal{L}(\mathcal{L}_t \times \mathcal{E}, \mathcal{L}_t ;\mathcal{C}_t)} 
			\Vert (\bar{v},\xi) \Vert_{\mathcal{L}_t \times \mathcal{E}}
			\leq c_{\mathcal{S}} \Vert (\bar{v},\xi) \Vert_{\mathcal{L}_t \times \mathcal{E}},
		\end{aligned}
	\end{equation*}
	where the last estimate is justified by \Cref{lem: EstS}. This takes care not only of \ref{termA} but also of \eqref{eq: PhiProof1}. We turn to \ref{termB} and estimate it from above by
	\begin{equation*}
		\begin{aligned}
			\Vert DD_v \mathcal{S}^t(\bar{v},\xi) \Vert_{\mathcal{L}(\mathcal{L}_t \times \mathcal{E}, \mathcal{L}_t ; \mathcal{C}_t)}
			\left( \Vert \delta_0^* ( \bar{w}(0)-w_0 ) \Vert_{\mathcal{E}}
			+ \Vert \Psi[\bar{w} - \Tilde{w},y] \Vert_{\mathcal{C}_t^*} 
			\right).
		\end{aligned}
	\end{equation*}
	Utilizing \Cref{lem: EstS} and \eqref{eq: estPsi1} it is further estimated by
	\begin{equation*}
		c_{\mathcal{S}} \left( \Vert \bar{w} - \Tilde{w} \Vert_{\mathcal{C}_t} + c \left( \Vert y \Vert_{\mathcal{Y}_t} + \Vert \bar{w} - \Tilde{w} \Vert_{\mathcal{C}_t} \right) \right).
	\end{equation*}
	This concludes the treatment of \eqref{eq: PhiProof2} and the first estimate is shown.
	
	It remains to show that $D \Phi_t(h_0)^{-1}$ exists and satisfies the estimate. For $(z,u,q,\gamma,\eta) \in X_t$ and $(\eta^\prime,\gamma^\prime,f_1,f_2,f_3) \in Y_t$ we have 
	\begin{equation*}
		D\Phi_t(h_0) [(z,u,q,\gamma,\eta)] = (\eta^\prime,\gamma^\prime,f_1,f_2,f_3)
	\end{equation*}
	if and only if it holds $\eta = \eta^\prime,~\gamma = \gamma^\prime$ in addition to
	\begin{equation}\label{eq: OptSysLin}
		\begin{aligned}
			z &= D_v \mathcal{S}^t (0,0)[u]
			+ D_\xi \mathcal{S}^t (0,0)[\eta] - f_1,\\
			q &= D_v \mathcal{S}^t (0,0)^* \left( \delta_0^* z(0) - \alpha \Psi[z,\gamma^\prime] \right) - f_2,\\
			u &= - q + f_3.
		\end{aligned}
	\end{equation}
	Note that one obtains \eqref{eq: OptSysLin} by considering \eqref{eq: OptSysGLQR} in the appendix and setting $\xi = 0$, $\bar{v} = 0$, $y = 0$, $f = D_\xi \mathcal{S}^t(0,0) [\eta] - f_1$, $\gamma = \gamma^\prime$,  $L_1 = -f_2$, and $L_2 = -f_3$. Therefore \Cref{prop: GLQR} yields the unique solvability of \eqref{eq: OptSysLin} and hence the invertibility of $D\Phi_t(h_0)$. The asserted estimate for the inverse follows from \Cref{prop: GLQR} together with \Cref{lem: EstS}.
\end{proof}
The foregoing considerations allow the application of a specific formulation of the implicit function theorem presented in \cite{Hol70}. This version of the well-known theorem ensures time independence of the involved neighborhoods.
\begin{lemma}\label{lem: ImplFct}
	Let $t \in(0,T]$ be arbitrary. Then there exist $\delta,\epsilon>0$ independent of $t$ and three $C^\infty$ functions $\mathcal{W}_t$, $\mathcal{U}_t$ and $\mathcal{P}_t$ such that for any $(\xi,y) \in \mathcal{D}_{\delta}^t$ the triple given by
	$(w,v,p) = (\mathcal{W}_t(\xi,y), \mathcal{U}_t(\xi,y), \mathcal{P}_t(\xi,y))$ 	
	is the unique solution to
	\begin{equation*}
		\Phi_t(w,v,p,y,\xi) = (\xi,y,0,0,0),~~~~~
		\Vert (w,v,p) \Vert_{\mathcal{C}_t \times \mathcal{L}_t \times \mathcal{L}_t} \leq \epsilon.
	\end{equation*}
\end{lemma}
\begin{proof}
	The proof is conducted by an application of \cite{Hol70}. The assumptions denoted by (i)-(vii) are checked using the notation given in the reference. We define $P\colon Y_t \times X_t \rightarrow Y_t$ as $P(s,h) = \Phi_t(h) - s$ and set $h_0$ as in \Cref{lem: estPhi}.
	\begin{enumerate}[label=(\roman*)]
		\item It holds $P(0,h_0) = \Phi_t(\Tilde{w},0,0,0,0) = 0$.
		\item Further $D_h P(s,h) = D\Phi_t(h)$ exists and is of class $C^\infty$. This is due to the smoothness of $\mathcal{S}^t$ established in \Cref{lem: SmoothS}.
		\item We set $k_1 = c_{\Phi}$, where $c_{\Phi}$ is the constant from \Cref{lem: estPhi}. By the same lemma we obtain existence of the inverse and the bound
		\begin{equation*}
			\Vert D_hP(0,h_0)^{-1} \Vert_{\mathcal{L}(Y_t;X_t)}
			= \Vert D \Phi_t(h_0)^{-1} \Vert_{\mathcal{L}(Y_t;X_t)}
			\leq k_1. 
		\end{equation*}
		\item Define $\epsilon = \min ( \delta_2,(4c_{\Phi}^2)^{-1} )$, $\delta = (2c_{\Phi})^{-1} \epsilon$ and
		\begin{equation*}
			S = \{ (s,h) \in Y_t \times X_t \colon \Vert s \Vert_{Y_t} < \delta,~ \Vert h - h_0 \Vert_{X_t} < \epsilon \}.
		\end{equation*}
		\item We define $g_1(x,y) = c_{\Phi} y$. Due to \Cref{lem: estPhi} for $(s,h) \in S$ it holds
		\begin{equation*}
			\begin{aligned}
				\Vert D_h P(s,h) - D_h P(0,h_0) \Vert_{\mathcal{L}(X_t;Y_t)}
				&= \Vert D \Phi(h) - D \Phi(h_0) \Vert_{\mathcal{L}(X_t;Y_t)}\\
				\leq & c_{\Phi} \Vert h - h_0 \Vert_{X_t} 
				= g_1(\Vert s \Vert_{Y_t}, \Vert h - h_0 \Vert_{X_t}).
			\end{aligned}
		\end{equation*}
		\item Define $g_2(x) = x$. Then
		\begin{equation*}
			\Vert P(s,h_0) \Vert_{Y_t}
			= \Vert \Phi_t(h_0) - s \Vert_{Y_t}
			= \Vert s \Vert_{Y_t}
			= g_2(\Vert s \Vert_{Y_t}).
		\end{equation*}
		\item Set $\alpha = \frac{1}{2}$. Then
		\begin{equation*}
			\begin{aligned}
				k_1 g_1(\delta,\epsilon)
				= c_{\Phi}^2 \epsilon \leq \frac{1}{4} < \alpha < 1
				~~~~~~\text{and}~~~~~~
				k_1 g_2(\delta) 
				= c_{\Phi} \delta = \frac{1}{2} \epsilon = \epsilon (1 - \alpha).
			\end{aligned}
		\end{equation*}
	\end{enumerate}
	Since all assumptions are fulfilled, the theorem yields an operator 
	\begin{equation*}
		F \colon \{ s \in Y_t \colon \Vert s \Vert_{Y_t} < \delta \}
		\rightarrow
		\{ h \in X_t \colon \Vert h - h_0 \Vert_{X_t} \leq \epsilon \}
	\end{equation*}
	such that $P(s,F(s)) = 0 $ which is equivalent to $\Phi_t(F(s)) = s $, where $F(s)$ is the unique element with this property lying in $\{ h \in X_t \colon \Vert h - h_0 \Vert_{X_t} \leq \epsilon \}$. Now for any $(\xi,y)$ satisfying $\Vert (\xi,y) \Vert_{\mathcal{E} \times \mathcal{Y}_t} < \delta$ we set $\mathcal{W}_t(\xi,y) = w$, $\mathcal{U}_t(\xi,y) = u$ and $\mathcal{P}_t(\xi,y) = p$ where $(w,u,p)$ is uniquely determined via
	$F(\xi,y,0,0,0) = (w,u,p,y,\xi)$.	
	
	We note that strictly speaking the theorem presented in \cite{Hol70} only yields continuity of $F$. However, comparing the assumptions and the proof with the ones of more classical versions (see for example \cite[Theorem 4.E]{Zei95AMS109}) shows that the neighborhoods are only adjusted within the application of the fixed point argument. The part of the proof that shows that differentiability of order $k$ of $P$ carries over to $F$ does not alter the neighborhoods at all. Hence the assertion of higher regularity holds for the chosen constants if $P$ is of class $C^\infty$. With \Cref{lem: SmoothS} this regularity is ensured for $\mathcal{S}^t$ and carries over to $\Phi_t$ and $P$. 
\end{proof}	
The constructed mappings $\mathcal{W}_t$, $\mathcal{U}_t$, and $\mathcal{P}_t$ map sufficiently small data $(\xi,y)$ to a unique state trajectory, control, and adjoint state satisfying a norm bound and the optimality condition established for \eqref{eq: OCP}. The following proposition ensures that they are actually the unique optimal triple.
\begin{proposition}
	Let $t \in (0,T]$ be arbitrary and let $\delta$, $\epsilon$ be the constants from \Cref{lem: ImplFct}. Then there exists a constant $\delta_3 \in (0,\min\left( \delta,\epsilon, \delta_2 \right)]$ such that for any $(\xi,y) \in \mathcal{D}_{\delta_3}^t$ the associated optimal control problem \eqref{eq: OCP} admits exactly one solution. The optimal triple is given by $\mathcal{W}_t(\xi,y)$, $\mathcal{U}_t(\xi,y)$, and $\mathcal{P}_t(\xi,y)$. 
\end{proposition}
\begin{proof}
	For now set $\delta_3 = \min\left( \delta,\epsilon, \delta_2 \right)$ and assume $\Vert (\xi,y) \Vert_{\mathcal{E} \times \mathcal{Y}_t} < \delta_3$. According to \Cref{lem: ImplFct} the triple $(\mathcal{W}_t(\xi,y), \mathcal{U}_t(\xi,y), \mathcal{P}_t(\xi,y))$ is the only one that satisfies the optimality condition established for \eqref{eq: OCP} and lies in 
	\begin{equation*}
		\mathcal{M} = \{ (w,v,p) \in \mathcal{C}_t \times \mathcal{L}_t \times \mathcal{L}_t \colon
		\Vert w - \Tilde{w} \Vert_{\mathcal{C}_t}, \Vert v \Vert_{\mathcal{L}_t}, \Vert p \Vert_{\mathcal{L}_t} \leq \epsilon \}.
	\end{equation*}
	Due to \Cref{prop: exMin} there exists at least one minimizer and further any minimizer $(\bar{w},\bar{v})$ satisfies 
	\begin{equation*}
		\max \left( \Vert \bar{w} - \Tilde{w} \Vert_{\mathcal{C}_t},
		\Vert \bar{v} \Vert_{\mathcal{L}_t} \right) 
		\leq c_\mathrm{Min} \max\left( \Vert \xi \Vert_\mathcal{E}, \Vert y \Vert_{\mathcal{Y}_t} \right)
		< c_\mathrm{Min} \delta_3.
	\end{equation*}
	With \Cref{lem: estAdj} it follows that the associated adjoint $p$ satisfies
	\begin{equation*}
		\Vert p \Vert_{\mathcal{L}_t}
		\leq c_p \max\left( \Vert \xi \Vert_\mathcal{E}, \Vert y \Vert_{\mathcal{Y}_t} \right)
		< c_p \delta_3.
	\end{equation*}
	Hence a decrease of $\delta_3$ ensures that any minimizing triple lies in $\mathcal{M}$. Since the only element of $\mathcal{M}$ satisfying the optimality condition is given by $(\mathcal{W}_t(\xi,y), \mathcal{U}_t(\xi,y), \mathcal{P}_t(\xi,y))$ and existence of a minimizer was shown, the assertion holds.
\end{proof}
We finally obtain regularity of the value function in space and output.
\begin{corollary}\label{cor: ValFunReg}
	For any fixed $t \in (0,T]$ and $(\xi,y) \in \mathcal{D}_{\delta_3}^t$ the value function is given as
	\begin{equation}\label{eq: ValFunSmooth}
		\mathcal{V}(t,\xi,y) = J(\mathcal{W}_t(\xi,y),\mathcal{U}_t(\xi,y);t,y).
	\end{equation}
	In particular the value function is of class $C^\infty$ with respect to $(\xi,y)$ in $\mathcal{D}_{\delta_3}^t$.
\end{corollary}
In the following we derive bounds for the partial derivatives of $\mathcal{W}_t$, $\mathcal{U}_t$, and $\mathcal{P}_t$ which will then be carried over to derivatives of $\mathcal{V}$. While such Fr\'echet derivatives are by definition linear bounded operators, we show that the bounds are in fact uniform in $t$. A linearization of $\Phi_t$ characterizes $D\mathcal{W}_t$ , $D\mathcal{U}_t$, $D\mathcal{P}_t$ as optimal triples associated with linear quadratic optimal control problems. As such they satisfy the desired estimates. A general linear quadratic OCP is analyzed in the appendix. Here we show that the derivatives of $\mathcal{W}_t$, $\mathcal{U}_t$, $\mathcal{P}_t$ solve problems of such type. 
\begin{lemma}\label{lem: DerSolLQR}
	Let $t \in (0,T]$ be fixed and let $(\xi,y) \in \mathcal{D}_{\delta_3}^t$ and $\mu \in \mathcal{Y}_t$. For $i = 1,2,3$ let $\eta_i \in \mathcal{E}$. There exists $c_\mathcal{W} > 0$ independent of $t$ such that for $i = 1,2,3$ it holds
	\begin{equation*}
		\begin{aligned}
			&\max \left( \left\Vert D_{\xi^i}^i \mathcal{W}_t(\xi,y) [\eta_1,...,\eta_i] \right\Vert_{\mathcal{C}_t}, \right.
			\left\Vert D_{\xi^i}^i \mathcal{U}_t(\xi,y) [\eta_1,...,\eta_i] \right\Vert_{\mathcal{L}_t}, 
			\left.
			\left\Vert D_{\xi^i}^i \mathcal{P}_t(\xi,y) [\eta_1,...,\eta_i] \right\Vert_{\mathcal{L}_t} \right)\\  
			&\leq c_\mathcal{W} \prod_{k=1}^i \Vert \eta_k \Vert_\mathcal{E}
		\end{aligned}
	\end{equation*}
	and for $j = 0,1,2$ it holds
	\begin{equation*}
		\begin{aligned}
			&\max \left( \left\Vert D_{y \xi^j}^{1+j} \mathcal{W}_t(\xi,y) [\mu, \eta_1,...,\eta_j] \right\Vert_{\mathcal{C}_t}, \right. 
			\left\Vert D_{y \xi^j}^{1+j} \mathcal{U}_t(\xi,y) [\mu, \eta_1,...,\eta_j] \right\Vert_{\mathcal{L}_t},\\
			&\left.
			\left\Vert D_{y \xi^j}^{1+j} \mathcal{P}_t(\xi,y)
			[\mu, \eta_1,...,\eta_j] \right\Vert_{\mathcal{L}_t} \right)  
			\leq c_\mathcal{W} \Vert \mu \Vert_{\mathcal{Y}_t} \prod_{k=1}^j \Vert \eta_k \Vert_\mathcal{E}.
		\end{aligned}
	\end{equation*}
\end{lemma}
\begin{proof}
	Again we only show one of the estimates to illustrate the technique. The rest follow analogously in an iterative fashion. Without loss of generality we assume $\delta_3 < \delta^\prime$, where $\delta^\prime$ is the constant introduced in the discussion of the general linear quadratic optimal control problem in the appendix, cf. \Cref{prop: GLQR}. By construction it holds
	\begin{equation}\label{eq: Phi}
		\Phi_t(\mathcal{W}_t(\xi,y),\mathcal{U}_t(\xi,y),\mathcal{P}_t(\xi,y),y,\xi)
		= (\xi,y,0,0,0).
	\end{equation}
	Taking a derivative with respect to $\xi$ in direction $\eta$ and applying the chain rule yields that
	\begin{equation*}
		(D_\xi \mathcal{W}_t (\xi,y) [\eta], D_\xi \mathcal{U}_t (\xi,y) [\eta], D_\xi \mathcal{P}_t (\xi,y) [\eta] )
	\end{equation*}
	satisfies the optimality system \eqref{eq: GLQR} associated with a linear quadratic optimal control problem analyzed in the appendix. The specific system is obtained by setting $\xi = \xi$, $\bar{v} = \mathcal{U}_t(\xi,y)$, $y = y$, $f = D_\xi \mathcal{S}^t(\bar{v},\xi) [\eta]$, $\gamma = 0$, $L_1 = (D_{v\xi}^2 \mathcal{S}(\bar{v},\xi) [\eta])^* (\delta_0^* (\bar{w}(0) - w_0) - \alpha \Psi[\bar{w} - \Tilde{w},y])$, and $L_2 = 0$. The assertion follows using standard estimates together with \Cref{prop: SolStateEq}, \Cref{lem: EstS}, and \Cref{prop: GLQR}.
\end{proof}
We are now able to derive norm bounds for the partial derivatives of $\mathcal{V}$.
\begin{lemma}\label{lem: ValFunDerEst}
	Let $t \in (0,T]$, $(\xi,y) \in \mathcal{D}_{\delta_3}^t$ and $\mu \in \mathcal{Y}_t$. For $i = 1,2,3,$ let $\eta_i \in \mathcal{E}$. There exists $c_\mathcal{V} > 0$ such that for $i=1,2,3$ and $j =1,2$ it holds
	\begin{equation*}
		\begin{aligned}
			\left\vert D_{\xi^i}^i \mathcal{V}(t,\xi,y) [\eta_1,...,\eta_i] \right\vert
			&\leq c_\mathcal{V} \prod_{k=1}^i \Vert \eta_k \Vert_\mathcal{E},\\
			\left\vert D_{y \xi^j}^{1+j} \mathcal{V}(t,\xi,y)[\mu,\eta_1,...,\eta_j] \right\vert
			&\leq c_\mathcal{V} \, \Vert \mu \Vert_{\mathcal{Y}_t} \prod_{k=1}^j \Vert \eta_k \Vert_\mathcal{E}.
		\end{aligned}
	\end{equation*}
\end{lemma}
\begin{proof}
	The proof follows by taking appropriate derivatives of \eqref{eq: ValFunSmooth} before applying standard estimates together with the estimates from \Cref{prop: SolStateEq} and \Cref{lem: DerSolLQR}.	
\end{proof}
%
\section{Properties of the second derivative}\label{sec: CharSecDer}
%
This section is dedicated to the analysis of the second order spatial derivative of the value function, i.e., the second derivative of the mapping
\begin{equation*}
	\left\{ \xi \in \mathcal{E} ~\colon~ \Vert \xi \Vert_{\mathcal{E}} < \delta_3 \right\} 
	\longrightarrow \mathbb{R},
	~~~~~~~~
	\xi \mapsto \mathcal{V}(t,\xi,y),
\end{equation*}
where $t \in [0,T]$ and $y \in \mathcal{Y}_t$ satisfying $\Vert y \Vert_{\mathcal{Y}_t} < \delta_3$ are fixed. We first clarify some notation. By definition of the Fr\'echet derivative we have $D_\xi \mathcal{V}(t,\xi,y) \in \mathcal{L}(\mathcal{E};\mathbb{R}) = \mathcal{E}^*$ and $D_{\xi\xi}^2 \mathcal{V}(t,\xi,y) \in \mathcal{L}(\mathcal{E};\mathcal{E}^*)$. We apply Riesz to identify the dual of $\mathcal{E}$ with $\mathcal{E}$ and obtain $D_\xi \mathcal{V}(t,\xi,y) \in \mathcal{E}$ and $D_{\xi\xi}^2 \mathcal{V}(t,\xi,y) \in \mathcal{L}(\mathcal{E})$. Reminiscent of the finite-dimensional case $D_\xi \mathcal{V}$ and $D_{\xi\xi}^2 \mathcal{V}$ will be referred to as the \textit{gradient} and the \textit{Hessian} of the value function.
\begin{remark}
	Since we have $\mathcal{E} = H^1_0(\Omega) \times L^2(\Omega)$, we implicitly identify the dual of $H^1_0(\Omega)$ with the space itself. While this is contrary to the common identification $ H^1_0(\Omega)^* \simeq H^{-1}(\Omega)$, it exactly fits the scalar products appearing in the following arguments.  
\end{remark}
First we show that locally the second derivative is an isomorphism before deploying arguments based on Riccati equations to analyze the time regularity of its inverse.
%
\subsection{Local coercivity}\label{subsec: LocCoerc}
%
We first show that for all $t \in [0,T]$ the bilinear form associated with the Hessian evaluated along the model is coercive. For sufficiently small data this property is then carried over to $D_{\xi\xi}^2\mathcal{V}(t,\xi,y)$ using continuity arguments. Since all results are clear for $t = 0$, we only discuss $t \in (0,T]$.
The following lemma characterizes the derivative along the model in terms of $\mathcal{W}_t$ and $\mathcal{U}_t$.
\begin{lemma}\label{lem: CharSecDerV}
	For all $t \in (0,T]$ and $\eta \in \mathcal{E}$ satisfying $\Vert \xi \Vert_{\mathcal{E}} < \delta_3$ it holds 
	\begin{equation*}
		\left( D_{\xi\xi}^2 \mathcal{V}(t,0,0) \, \eta, \eta \right)_{\mathcal{E}}
		= \Vert D_\xi \mathcal{W}_t (0,0) [\eta] \, (0) \Vert_\mathcal{E}^2
		+ \Vert D_\xi \mathcal{U}_t (0,0) [\eta] \Vert_{\mathcal{L}_t}^2 
		+ \Vert C D_\xi \mathcal{W}_t (0,0) [\eta] \Vert_{\mathcal{Y}_t}^2. 
	\end{equation*}
\end{lemma}
\begin{proof}
	The assertion follows directly from \Cref{cor: ValFunReg} and the chain rule.
\end{proof}
This characterization enables the proof that the Hessian is in fact coercive.
\begin{proposition}\label{prop: modelCoerc}
	There exists a constant $M_0 > 0$ such that for all $t \in [0,T]$ and $z \in \mathcal{E}$ it holds
	\begin{equation*}
		\left( D_{\xi\xi}^2 \mathcal{V}(t,0,0) \, \eta, \eta \right)_\mathcal{E} 
		\geq M_0 \, \Vert \eta \Vert_\mathcal{E}^2.
	\end{equation*}
\end{proposition}
\begin{proof}
	In the proof of \Cref{lem: DerSolLQR} it is shown that the tuple 
	\begin{equation*}
		(w^*,v^*) = \left( D_\xi \mathcal{W}_{t}(0,0)[\eta], D_\xi \mathcal{U}_{t}(0,0)[\eta] \right)
	\end{equation*}
	solves a linear quadratic optimal control problem. Using \Cref{lem: SmoothS} we specify that $(w^*,v^*)$ is the unique minimizer of 
	\begin{equation}\label{eq: linOCP}
		\begin{aligned}
			\min J_0(w,v) &= \Vert w(0) \Vert_\mathcal{E}^2
			+ \Vert v \Vert_{\mathcal{L}_t}^2 
			+ \alpha \Vert C w \Vert_{\mathcal{Y}_t}^2~~~~~~\text{subject to} \\
			w(s) = & ~e^{-\mathcal{A}(t-s)} \eta
			- \int_s^{t} e^{-\mathcal{A}(\tau-s)} \begin{bmatrix} 0 \\ -3 \Tilde{w}_1^2(\tau) w_1(\tau) + v(\tau) \end{bmatrix} \, \mathrm{d}\tau~~~\forall s \in [0,t].
		\end{aligned}
	\end{equation}
	In the spirit of \Cref{rem: DefSol} we obtain for all $s \in [0,t]$
	\begin{equation*}
		w^*(s) = e^{\mathcal{A}s} w^*(0)
		+ \int_0^s e^{\mathcal{A}(s-\tau)} \begin{bmatrix} 0 \\ -3 \Tilde{w}_1^2(\tau) w^*_1(\tau) + v^*(\tau) \end{bmatrix} \, \mathrm{d}\tau.
	\end{equation*}
	Since the equation holds in particular for $s = t$, using standard estimates and Gronwall's inequality we obtain 
	\begin{equation*}
		\Vert \eta \Vert_{\mathcal{E}}^2 = \Vert w^*(t) \Vert_{\mathcal{E}}^2
		\leq c \left( \Vert w^*(0) \Vert_\mathcal{E}^2 + \Vert v^* \Vert_{\mathcal{L}_t}^2 \right)
		\leq c \, J_0(w^*,v^*)
		= c \left( D_{\xi\xi}^2 \mathcal{V}(t,0,0) \eta, \eta \right)_\mathcal{E},
	\end{equation*}
	where the last equality holds due to \Cref{lem: CharSecDerV}. Note that due to $t \in (0,T]$ the constant $c > 0$ can be chosen such that it depends on $T$ but not on $t$.
\end{proof}

Using continuity arguments the coercivity is extended to a neighborhood around the model. 
\begin{proposition}\label{prop: D2Coerc}
	There exists $\delta_4 \in (0,\delta_3]$ such that for every $t \in [0,T]$ and $(\xi,y) \in \mathcal{D}_{\delta_4}^t$ the bilinear form associated with $D_{\xi\xi}^2 \mathcal{V}(t,\xi,y) \in \mathcal{L}(\mathcal{E})$ is coercive, i.e., there exists $M>0$ such that for all $t \in [0,T]$ and $\eta \in \mathcal{E}$ it holds
	\begin{equation*}
		\left( D_{\xi\xi}^2 \mathcal{V}(t,\xi,y) \, \eta,\eta \right)_\mathcal{E} \geq M \Vert \eta \Vert_\mathcal{E}^2.
	\end{equation*}
	In particular $D_{\xi\xi}^2 \mathcal{V}(t,\xi,y)$ is an isomorphism on $\mathcal{E}$ and
	\begin{equation*}
		\Vert D_{\xi\xi}^2 \mathcal{V}(t,\xi,y)^{-1} \Vert_{\mathcal{L}(\mathcal{E})}
		\leq M^{-1}.
	\end{equation*}
\end{proposition}

\begin{proof}
	For now set $\delta_4 = \delta_3$. For fixed $t \in [0,T]$ and $(\xi,y) \in \mathcal{D}_{\delta_4}^t$ we identify the second derivative with its associated bilinear form and use Taylor's theorem to estimate
	\begin{equation*}
		\begin{aligned}
			&\Vert D_{\xi\xi}^2 \mathcal{V}(t,\xi,y) - D_{\xi\xi}^2 \mathcal{V}(t,0,0) \Vert_{\mathcal{L}(\mathcal{E},\mathcal{E};\mathbb{R})}\\
			&= \left\Vert \int_0^1 D_{\xi^3}^3 \mathcal{V}(t,\tau\xi,\tau y) [\xi] 
			+ D_{y \xi^2}^3 \mathcal{V}(t,\tau\xi,\tau y)[y] \, \mathrm{d} \tau \right\Vert_{\mathcal{L}(\mathcal{E},\mathcal{E};\mathbb{R})}\\
			&\leq \int_0^1 \Vert D_{\xi^3}^3 \mathcal{V}(t,\tau\xi,\tau y) \Vert_{\mathcal{L}(\mathcal{E},\mathcal{E},\mathcal{E};\mathbb{R})}
			\Vert \xi \Vert_\mathcal{E}
			+ \Vert D_{y \xi^2}^3 \mathcal{V}(t,\tau\xi,\tau y) \Vert_{\mathcal{L}(\mathcal{Y}_t,\mathcal{E},\mathcal{E};\mathbb{R})}
			\Vert y \Vert_{\mathcal{Y}_t}
			\, \mathrm{d}\tau.
		\end{aligned}
	\end{equation*}
	With \Cref{lem: ValFunDerEst} we get
	\begin{equation}\label{eq: EstSecDer}
		\Vert D_{\xi\xi}^2 \mathcal{V}(t,\xi,y) - D_{\xi\xi}^2 \mathcal{V}(t,0,0) \Vert_{\mathcal{L}(\mathcal{E},\mathcal{E};\mathbb{R})}
		\leq c_\mathcal{V} \Vert \xi \Vert_\mathcal{E} + c_\mathcal{V} \Vert y \Vert_{\mathcal{Y}_t} 
	\end{equation}
	and hence with an appropriate decrease of $\delta_4 $ it holds
	\begin{equation*}
		\Vert D_{\xi\xi}^2 \mathcal{V}(t,\xi,y) - D_{\xi\xi}^2 \mathcal{V}(t,0,0) \Vert_{\mathcal{L}(\mathcal{E},\mathcal{E};\mathbb{R})}
		\leq 2 c_\mathcal{V} \delta_4
		\leq \tfrac{1}{2} M_0.
	\end{equation*}
	Together with \Cref{prop: modelCoerc} we obtain
	\begin{equation*}
		\begin{aligned}
			\left( D_{\xi \xi}^2 \mathcal{V}(t,\xi,y) \eta,\eta \right)_\mathcal{E}
			&= \left( D_{\xi \xi}^2 \mathcal{V}(t,0,0) \eta,\eta \right)_\mathcal{E}
			+ \left( \left[ D_{\xi \xi}^2 \mathcal{V}(t,\xi,y) - D_{\xi \xi}^2 \mathcal{V}(t,0,0) \right] \eta,\eta \right)_\mathcal{E}\\
			& \geq M_0 \Vert \eta \Vert_\mathcal{E}^2
			- \left\Vert D_{\xi\xi}^2 \mathcal{V}(t,\xi,y) - D_{\xi \xi}^2 \mathcal{V}(t,0,0) \right\Vert_{\mathcal{L}(\mathcal{E},\mathcal{E};\mathbb{R})} \Vert \eta \Vert_\mathcal{E}^2
			\geq \tfrac{1}{2} M_0 \Vert \eta \Vert_\mathcal{E}^2
		\end{aligned}
	\end{equation*}
	and coercivity is shown. The second part of the assertion follows with the Lax-Milgram Lemma, cf., \cite[Thm.~1.1.3 {\&} Rem.~1.1.3]{Cia78}.
\end{proof}
%
\subsection{Strong continuity in time}\label{subsec: StrCont}
%
In this subsection we characterize the Hessian along the model as the solution of an integral Riccati equation ensuring that it is a strongly continuous operator. This property is then transferred to the inverted Hessian.
The following statement and its proof rely on the concept of \textit{evolution operators}. For an introduction of the topic the reader is referred to \cite[P.~II, Ch.~3.5]{BenEtAl07}.

\begin{proposition}
	The function $[0,T] \mapsto \mathcal{L}(\mathcal{E})$, $t \mapsto D_{\xi\xi}^2 \mathcal{V}(T-t,0,0)$ is given as the unique strongly continuous, self-adjoint, non-negative solution $P(\cdot)$ of the \textit{integral Riccati equation}
	\begin{equation}\label{eq: IntRicc}
		P(t) 
		= U^*(T,t) ~ U(T,t)
		+ \int_t^T U^*(s,t) \left[ C^* C - P(s) B B^* P(s) \right] U(s,t) \, \mathrm{d}s,
	\end{equation}
	where $B \in \mathcal{L}(L^2(\Omega);\mathcal{E})$ is defined as $Bv = \begin{bmatrix} 0 \\ - v \end{bmatrix}$ and $U(\cdot,\cdot)$ is the \textit{evolution operator} associated with $-\mathcal{A} + \begin{bmatrix} 0 & 0 \\ 3 \Tilde{w}_1^2(T-\cdot) \, \mathrm{Id} & 0 \end{bmatrix}.$ 
\end{proposition}
\begin{proof}
	Let $t \in (0,T] $ and $\eta \in \mathcal{E}$ be fixed. From \Cref{lem: CharSecDerV} and the proof of \Cref{prop: modelCoerc} it follows that 
	\begin{equation*}
		\left( D_{\xi\xi}^2 \mathcal{V}(T-t,0,0) \eta, \eta \right)_\mathcal{E}
		= J_1(w^*,v^*),
	\end{equation*}
	where
	\begin{equation}\label{eq: J1}
		J_1(w,v) = \Vert w(0) \Vert_\mathcal{E}^2
		+ \Vert v \Vert_{\mathcal{L}_{T-t}}^2 
		+ \alpha \Vert C w \Vert_{\mathcal{Y}_{T-t}}^2
	\end{equation}
	and $(w^*,v^*)$ is the minimizer of $J_1$ under the constraint
	\begin{equation}\label{eq: Constr1}
		w(s) = e^{-\mathcal{A}(T-t-s)} \eta
		- \int_s^{T-t} e^{-\mathcal{A}(\tau-s)} \begin{bmatrix} 0 \\ -3 \Tilde{w}_1^2(\tau) w_1(\tau) + v(\tau) \end{bmatrix} \, \mathrm{d}\tau
		~~~~~~\forall s \in [0,T-t].
	\end{equation}
	In the following we equivalently transform \eqref{eq: J1}-\eqref{eq: Constr1} into a forward problem in order to apply results from \cite{Fla86}.
	
	For a given function $f$ defined on $[0,T-t]$ we define the time-transformed and shifted function $\overleftarrow{f}$ via $\overleftarrow{f}(s) \coloneqq f(T-s)$. We find that $(w,v)$ satisfies \eqref{eq: Constr1} if and only if
	\begin{equation}\label{eq: Constr2}
		\overleftarrow{w}(s) 
		= e^{- \mathcal{A}(s-t)} \eta
		+ \int_t^s e^{-\mathcal{A}(s- \tau)} \begin{bmatrix} 0 \\ 3 \overleftarrow{\Tilde{w}}_1^2(\tau) \overleftarrow{w}_1(\tau) - \overleftarrow{v}(\tau) \end{bmatrix} \, \mathrm{d} \tau
		~~~~~~\forall s \in [t,T].
	\end{equation}
	Note that due to the regularity of $\Tilde{w}$ the operator $F(s) = \begin{bmatrix} 0 & 0 \\ 3 \overleftarrow{\Tilde{w}}_1^2(s) \mathrm{Id} & 0 \end{bmatrix}$ is strongly continuous. Hence \eqref{eq: Constr2} is equivalent to 
	\begin{equation}\label{eq: Constr3}
		\overleftarrow{w}(s) = 
		U(s,t) \eta + \int_t^s U(s,\tau) \begin{bmatrix} 0 \\ - \overleftarrow{v}(\tau) \end{bmatrix} \, \mathrm{d}\tau~~~\forall s \in [t,T],
	\end{equation}
	where $U(\cdot,\cdot)$ is the evolution operator associated with $-\mathcal{A} + F$. 
	Defining the cost functional 
	\begin{equation}
		J_2(w,v) = \Vert w(T) \Vert_\mathcal{E}^2
		+ \int_t^T \Vert v(s) \Vert_{L^2(\Omega)}^2 
		+ \alpha \Vert C w(s) \Vert_{Y}^2 \, \mathrm{d}s
	\end{equation}
	we observe $J_1(w,v) = J_2(\overleftarrow{w},\overleftarrow{v})$ and conclude that $(\overleftarrow{w^*}, \overleftarrow{v^*})$ minimizes $J_2$ under the constraint \eqref{eq: Constr3}. Further it holds
	\begin{equation*}
		\left( D_{\xi\xi}^2 \mathcal{V}(T-t,0,0) \eta, \eta \right)_\mathcal{E}
		= J_2(\overleftarrow{w^*},\overleftarrow{v^*}).
	\end{equation*}
	We apply the results from \cite[Sec.~3.2]{Fla86} to the control problem given by $J_2$ and \eqref{eq: Constr3}. Due to the arguments made in \cite[Sec.~2]{Gib79} and the separability of $\mathcal{E}$ no strong continuity of $U^*$ is required to do so. It follows
	\begin{equation*}
		J_2(\overleftarrow{w^*},\overleftarrow{v^*})
		= \left( P(t) \eta,\eta \right)_\mathcal{E},
	\end{equation*}
	where $P$ is the unique strongly continuous, self-adjoint, non-negative solution of \eqref{eq: IntRicc}. Since $t$ and $\eta$ were chosen arbitrarily, it follows that for all $\eta \in \mathcal{E}$ and $t \in (0,T]$ it holds
	\begin{equation*}
		\left(D_{\xi\xi}^2 \mathcal{V}(T-t,0,0) \eta,\eta \right)_\mathcal{E}
		= \left( P(t)\eta,\eta \right)_\mathcal{E}.
	\end{equation*}
	With the self-adjointness of the two operators and the \textit{polarization identity} it follows $D_{\xi\xi}^2 \mathcal{V}(T-t,0,0) = P(t)$ and the assertion is shown.
\end{proof}
The section is concluded by a proof that the strong continuity and uniform coercivity of $D_{\xi\xi}^2 \mathcal{V}(t,0,0)$ imply strong continuity of $D_{\xi\xi}^2 \mathcal{V}(t,0,0)^{-1}$.
\begin{corollary}\label{cor: InvStrCont}
	The mapping $[0,T] \mapsto \mathcal{L}(\mathcal{E})$, $t \mapsto D_{\xi\xi}^2 \mathcal{V}(t,0,0)^{-1}$ is strongly continuous.
\end{corollary}
\begin{proof}
	For convenience denote only within this proof $ A(t) = D_{\xi\xi}^2 \mathcal{V}(t,0,0)$ and let $ t \in [0,T]$ and $\tau \in \mathbb{R}$ such that $t + \tau \in [0,T]$. For any $\eta \in \mathcal{E}$ it holds
	\begin{equation}\label{eq: StrCon}
		\begin{aligned}
			\Vert A^{-1}(t + \tau)x - A^{-1}(t)x \Vert_\mathcal{E}
			&\leq
			\Vert A^{-1}(t + \tau) \Vert_{\mathcal{L}(\mathcal{E})} 
			\Vert x - A(t+\tau) A^{-1}(t) x \Vert_{\mathcal{E}}\\
			&\leq \frac{1}{M} \Vert x - A(t+\tau) A^{-1}(t) x \Vert_{\mathcal{E}},
		\end{aligned}
	\end{equation}
	where \Cref{prop: D2Coerc} was used. Since $A^{-1}(t) x \in \mathcal{E}$ is independent of $\tau$, the strong continuity of $A(\cdot)$ implies
	\begin{equation*}
		A(t+\tau) A^{-1}(t) x \rightarrow A(t) A^{-1}(t) x = x,~~~ \tau \to 0
	\end{equation*}
	and the right hand side in \eqref{eq: StrCon} tends to zero for $\tau \to 0$. Then $A^{-1}(t+\tau) x \rightarrow A^{-1}(t)x,~\tau \to 0$.
\end{proof}
%
%
\section{Well-posedness of the Mortensen observer}\label{sec: WPMort}
%
With the preparations from the previous sections we are finally in a position to show that the Mortensen observer is well-defined as the minimizer of the value function and that the observer equation admits a solution.
%
\subsection{Existence of a unique minimizer of the value function}\label{subsec: MorExMin}
%
In this subsection we show that for sufficiently small $y$ and any fixed $t \in [0,T]$ the mapping $\xi \mapsto \mathcal{V}(t,\xi,y)$ admits a unique minimizer. Again the result is clear for $t = 0$ and we only consider $t \in (0,T]$.
First we characterize any minimizer as a root of the gradient that satisfies an upper bound and show uniqueness of the minimizer. This requires a result analogous to \Cref{prop: SolStateEq} for the corresponding forward problem. 
\begin{corollary}\label{cor: SolForwState}
	Let $\delta_1$ and $c_w$ be the constants from \Cref{prop: SolStateEq}. For $t \in (0,T]$ and
	$\left\Vert \xi \right\Vert_{\mathcal{E}} + \Vert v \Vert_{\mathcal{L}_t} < \delta_1$
	the equation
	\begin{equation*}
		w(s) = e^{\mathcal{A}s} \left( \Tilde{w} (0) + \xi \right)
		+\int_0^s e^{\mathcal{A}(s-\tau)} \begin{bmatrix} 0 \\ - w_1^3(\tau) + v(\tau) \end{bmatrix} \, \mathrm{d} \tau
		~~~~~
		\forall s \in [0,t]
	\end{equation*}
	admits exactly one solution $w$. It satisfies
	\begin{equation*}
		\left\Vert w - \Tilde{w} \right\Vert_{\mathcal{C}_t} 
		\leq 2 c_w \left( \Vert \xi \Vert_\mathcal{E} + \Vert v \Vert_{\mathcal{L}_t} \right).
	\end{equation*}
\end{corollary}
\begin{proof}
	The assertion can be shown analogously to \Cref{prop: SolStateEq}. 
	Since $\Vert e^{\mathcal{A}t} \Vert_{\mathcal{L}(\mathcal{E})} = \Vert e^{-\mathcal{A}t} \Vert_{\mathcal{L}(\mathcal{E})}$. The same constants $\delta_1$ and $c_w$ appear.	
\end{proof}
We now show that the vanishing gradient together with an upper bound is a necessary and sufficient condition to be a minimizer.
\begin{lemma}\label{lem: CharMin}
	There exists $\delta_5 \in (0,\delta_4]$ such that for any $t \in (0,T] $ and $y \in \mathcal{Y}_t$ satisfying $\Vert y \Vert_{\mathcal{Y}_t} < \delta_5$ it holds: The mapping $\xi \mapsto \mathcal{V}(t,\xi,y)$ admits at most one global minimizer and 
	\begin{equation*}
		\xi \in \mathcal{E} \text{ is a minimizer of } \mathcal{V}(t,\cdot,y)
		\iff
		D_\xi \mathcal{V}(t,\xi,y) = 0 \text{ and } \Vert \xi \Vert_\mathcal{E} < \delta_4.
	\end{equation*}
\end{lemma}
\begin{proof}
	We start by establishing an upper bound for the norm of any minimizer. Set $\delta_5 = \delta_4$ and let $\xi^* \in \mathcal{E}$ such that $\mathcal{V}(t,\xi^*,y) \leq \mathcal{V}(t,0,y)$. Then
	\begin{equation*}
		\mathcal{V}(t,\xi^*,y) 
		\leq \inf_{v} J(\mathcal{S}^t(v,0),v;t,y)
		\leq J(\mathcal{S}^t(0,0),0;t,y)
		= J\left(\Tilde{w},0;t,y\right)
		\leq \frac{\alpha}{2} \Vert y \Vert_{\mathcal{Y}_t}^2
		\leq \frac{\alpha}{2} \, \delta_5^2. 
	\end{equation*}
	Hence there exists $w^* \in \mathcal{C}_t$ and $v^* \in \mathcal{L}_t$ satisfying
	\begin{equation}\label{eq: starOne}
		J(w^*,v^*;t,y) = \frac{1}{2} \left\Vert w^*(0) - \Tilde{w}(0) \right\Vert_\mathcal{E}^2
		+ \frac{1}{2} \Vert v^* \Vert_{\mathcal{L}_t}^2 
		+ \frac{\alpha}{2} \left\Vert y - C\left( w^* - \Tilde{w} \right) \right\Vert_{\mathcal{Y}_t}^2  
		< \alpha \delta_5^2
	\end{equation}
	and $w^*$ solves the state equation controlled by $v^*$ which by the arguments made in \Cref{rem: DefSol} is equivalent to
	\begin{equation*}
		w^*(s) 
		= e^{\mathcal{A}s} \left( \Tilde{w}(0) + \left( w^*(0) - \Tilde{w}(0) \right) \right)
		+ \int_0^s e^{\mathcal{A}(s-\tau)} \begin{bmatrix} 0 \\ - {w^*}_1^3 + v^* \end{bmatrix} \, \mathrm{d}\tau
		~~~~~\forall s \in [0,t]. 
	\end{equation*}
	With \eqref{eq: starOne} we get
	\begin{equation*}
		\Vert w^*(0) - \Tilde{w}(0) \Vert_\mathcal{E}^2
		+ \Vert v^* \Vert_{\mathcal{L}_t}^2 
		\leq 2 \, J(w^*,v^*;t,y)
		\leq 2 \alpha \, \delta_5^2.
	\end{equation*}
	A decrease of $\delta_5$ ensures $\Vert w^*(0) - \Tilde{w}(0) \Vert_\mathcal{E} + \Vert v^* \Vert_{\mathcal{L}_t} < \delta_1$ and \Cref{cor: SolForwState} yields
	\begin{equation*}
		\Vert \xi^* \Vert_\mathcal{E}
		= \left\Vert w^*(t) - \Tilde{w}(t) \right\Vert_\mathcal{E}
		\leq \left\Vert w^* - \Tilde{w} \right\Vert_{\mathcal{C}_t}
		\leq 2 c_w \left( \Vert w^*(0) - \Tilde{w}(0) \Vert_\mathcal{E} + \Vert v^* \Vert_{\mathcal{L}_t} \right)
		< c \delta_5.
	\end{equation*}
	Yet another decrease of $\delta_5$ yields that any $\xi^*$ satisfying $\mathcal{V}(t,\xi^*,y) \leq \mathcal{V}(t,0,y)$ must satisfy $\Vert \xi^* \Vert_{\mathcal{E}} < \delta_4$
	or equivalently 
	\begin{equation}\label{eq: LowBound}
		\Vert \eta \Vert_\mathcal{E} \geq \delta_4
		\implies
		\mathcal{V}(t,\eta,y) > \mathcal{V}(t,0,y).
	\end{equation}
	Now let $\xi^*$ be a minimizer. It follows $\mathcal{V}(t,\xi^*,y) \leq \mathcal{V}(t,0,y)$ and hence $\Vert \xi^* \Vert_{\mathcal{E}} < \delta_4$. The fact that $D_\xi \mathcal{V}(t,\xi^*,y) = 0$ follows immediately.
	
	Assume now $\xi^*$ satisfies $\Vert \xi^* \Vert_\mathcal{E}< \delta_4$ and $D_\xi \mathcal{V}(t,\xi^*,y) = 0$. Further let $\eta \in \mathcal{E}$ be such that $\Vert \eta \Vert_\mathcal{E} < \delta_4$ and $\eta \neq \xi^*$. Applying Taylor's theorem yields
	\begin{equation*}
		\begin{aligned}
			\mathcal{V}(t,\eta,y)
			&= \mathcal{V}(t,\xi^*,y)
			+ \left( D_\xi \mathcal{V}(t,\xi^*,y), \eta - \xi^* \right)_\mathcal{E} \\
			&+ \int_0^1 (1-\tau)~ \left( D_{\xi \xi}^2 \mathcal{V}(t,\xi^* + \tau(\eta - \xi^*),y)(\eta - \xi^*),\eta - \xi^* \right)_\mathcal{E} \, \mathrm{d}\tau.		
		\end{aligned}
	\end{equation*}
	Due to the vanishing gradient and \Cref{prop: D2Coerc} it follows that for all $\eta \neq \xi^* $ satisfying $\Vert \eta \Vert_\mathcal{E} < \delta_4$ it holds
	\begin{equation*}
		\mathcal{V}(t,\eta,y) 
		\geq \mathcal{V}(t,\xi^*,y)
		+ M^{-1} \Vert \eta - \xi^* \Vert_\mathcal{E}^2
		> \mathcal{V}(t,\xi^*,y).
	\end{equation*}
	With \eqref{eq: LowBound} for all $\eta $ with $\Vert \eta \Vert_\mathcal{E} \geq \delta_4$ it holds
	\begin{equation*}
		\mathcal{V}(t,\xi^*,y) \leq \mathcal{V}(t,0,y) < \mathcal{V}(t,\eta,y).
	\end{equation*}
	It follows that $\xi^*$ is the unique global minimizer.
\end{proof}
Applying the implicit function theorem shows existence of a minimizer.
\begin{theorem}
	There exists $\delta_6 \in (0,\delta_5]$ such that for every $t \in (0,T]$ and $y \in \mathcal{Y}_t$ satisfying $\Vert y \Vert_{\mathcal{Y}_t} < \delta_6$ the mapping $ \xi \mapsto \mathcal{V}(t,\xi,y) $ admits exactly one minimizer $\xi^*$. It satisfies $\Vert \xi^* \Vert_\mathcal{E} < \delta_4$ and depends continuously on $y \in \mathcal{Y}_t$.
\end{theorem}
\begin{proof}
	Let $t \in (0,T]$ be fixed and for now set $\delta_6 = \delta_5$. Let $y$ satisfy the assumption. Due to \Cref{lem: CharMin} it suffices to show existence of $\xi^* \in \mathcal{E}$ satisfying $\Vert \xi^* \Vert_\mathcal{E} < \delta_4$ and $D_\xi \mathcal{V}(t,\xi^*,y) = 0$. To that end we apply the implicit function theorem \cite[Thm.]{Hol70} to the mapping
	\begin{equation*}
		P \colon \Omega^\prime \rightarrow \mathcal{E},~~~~~
		P(y,\xi) = D_\xi \mathcal{V}(t,\xi,y),
	\end{equation*}
	where $\Omega^\prime = \left\{ (y,\xi) \in \mathcal{Y}_t \times \mathcal{E} \colon \Vert (y,\xi) \Vert_{\mathcal{Y}_t \times \mathcal{E}} < \delta_4 \right\}$. Due to \Cref{cor: ValFunReg} $P$ is well-defined and continuous in $\Omega^\prime$. Further $(y_0,\xi_0) = (0,0) \in \Omega^\prime$. We adopt the notation of the reference while checking the assumptions (i)-(vii).
	\begin{enumerate}[label=(\roman*)]
		\item Note that $\mathcal{V}(t,\cdot,0)$ is non-negative and $\mathcal{V}(t,0,0) = 0$, implying that $\xi = 0$ is a minimizer. It follows $P(y_0,\xi_0) = D_\xi \mathcal{V}(t,0,0) = 0$.
		\item Utilizing \Cref{cor: ValFunReg} we get that the derivative $(y,\xi) \mapsto D_\xi P(y,\xi) = D_{\xi\xi}^2\mathcal{V}(t,\xi,y) $ exists and is continuous in $\Omega^\prime$.
		\item According to \Cref{prop: D2Coerc} it holds
		\begin{equation*}
			\Vert D_\xi P(y_0,\xi_0)^{-1} \Vert_{\mathcal{L}(\mathcal{E})}
			= \Vert D_{\xi\xi}^2 \mathcal{V}(t,0,0)^{-1} \Vert_{\mathcal{L}(\mathcal{E})}
			\leq M^{-1} \eqqcolon k_1.
		\end{equation*}
		Note that $M$ is independent of $t$. Hence the same holds for $k_1$.
		\item We define $S = \left\{ (y,\xi) \in \mathcal{Y}_t \times \mathcal{E} \colon \Vert y \Vert_{\mathcal{Y}_t} < \delta, \Vert \xi \Vert_\mathcal{E} < \epsilon \right\} \subset \mathcal{D}_{\delta_4}^t$, where
		\begin{equation*}
			\epsilon = \min ( \tfrac{1}{2} \delta_5, \tfrac{1}{2}( \tfrac{1}{2} + k_1 c_{\mathcal{V}} )^{-1} )
			~~~\text{and}~~
			\delta = \min ( \delta_5, ( 2 k_1 c_{\mathcal{V}}  )^{-1} \epsilon ).
		\end{equation*}	
		\item We define $g_1(a,b) = c_{\mathcal{V}} (a + b) $ and using \eqref{eq: EstSecDer}  for all $(y,\xi) \in S$ we obtain
		\begin{equation*}
			\begin{aligned}
				\Vert D_\xi P(y,\xi) - D_\xi P(y_0,\xi_0) \Vert_{\mathcal{L}(\mathcal{E})}
				= \Vert D_{\xi\xi}^2 \mathcal{V}(t,\xi,y) - D_{\xi\xi}^2 \mathcal{V}(t,0,0) \Vert_{\mathcal{L}(\mathcal{E})}
				\leq g_1(\Vert y \Vert_{\mathcal{Y}_t},\Vert \xi \Vert_\mathcal{E}).
			\end{aligned}
		\end{equation*}
		\item We define $g_2(a) = c_\mathcal{V} \, a$. Using Taylor's theorem and \Cref{lem: ValFunDerEst} for $(y,\xi) \in S$ it holds 
		\begin{equation*}
			\begin{aligned}
				\Vert P(y,\xi_0) \Vert_\mathcal{E}
				&= \Vert D_\xi \mathcal{V}(t,0,0) 
				+ \int_0^1 D_{y\xi}^2 \mathcal{V}(t,0,\tau y) [y] \, \mathrm{d} \tau \Vert_\mathcal{E}\\
				&\leq \int_0^1 \Vert D_{y\xi}^2 \mathcal{V}(t,0,\tau y) \Vert_{\mathcal{L}(\mathcal{Y}_t;\mathcal{E})} \, \mathrm{d}\tau
				~\Vert y \Vert_{\mathcal{Y}_t}
				\leq g_2(\Vert y \Vert_{\mathcal{Y}_t}). 
			\end{aligned}
		\end{equation*}
		\item We set $\alpha = \frac{1}{2} < 1$ and it follows
		\begin{equation*}
			\begin{aligned}
				k_1 g_1(\delta,\epsilon)
				&= k_1 c_\mathcal{V} (\delta + \epsilon)
				\leq k_1 c_\mathcal{V} \epsilon \, (2k_1 c_\mathcal{V})^{-1}
				+ k_1 c_\mathcal{V} \epsilon
				\leq (\tfrac{1}{2} + k_1 c_\mathcal{V} ) \epsilon 
				\leq \tfrac{1}{2},\\
				k_1 g_2(\delta)
				&= k_1 c_\mathcal{V} \delta 
				\leq \tfrac{\epsilon}{2}
				= \epsilon \, (1-\alpha).
			\end{aligned}
		\end{equation*}
	\end{enumerate}
	Now \cite[Thm.]{Hol70} yields existence of an operator $F$ such that for any $y$ satisfying $\Vert y \Vert_{\mathcal{Y}_t} < \delta $ it holds $\Vert F(y) \Vert_{\mathcal{E}} \leq \epsilon < \delta_5 \leq \delta_4$ and 
	\begin{equation*}
		0 = P(y,F(y)) = D_\xi \mathcal{V}(t,F(y),y).
	\end{equation*} 
	A possible decrease of $\delta_6$ ensures $\delta_6 \leq \delta$ and the assertion is shown.
\end{proof}
\begin{corollary}\label{cor: MorWD}
	For any $y \in \mathcal{Y}_T$ satisfying $\Vert y \Vert_{\mathcal{Y}_T} < \delta_6$ and $t \in [0,T]$ the Mortensen observer can be defined via 
	\begin{equation*}
		\widehat{w} \colon [0,T] \rightarrow \mathcal{E},~~~~~
		\widehat{w}(t) = \argmin\limits_{\xi \in \mathcal{E}} \mathcal{V}(t,\xi,y).
	\end{equation*}	
	For every $t \in [0,T]$ we have that $\widehat{w}(t)$ depends continuously on $y$ and further $\Vert \widehat{w}(t) \Vert_\mathcal{E} < \delta_4$. In particular $\widehat{w} \in L^\infty(0,T;\mathcal{E})$. 
\end{corollary}
%
\subsection{Local well-posedness of the observer equation}\label{subsec: MorObsEq}
%
In \cite{BreS24} we used similar strategies to show existence of an energy minimizing trajectory in the setting of ODEs. Utilizing regularity of the value function and optimal trajectories we proceeded to show that this minimizing trajectory also satisfies the observer equation. 
Unfortunately this strategy can no be applied in the PDE setting discussed here, cf., \Cref{rem: OptTrajNotClassical}. We can, however, show that the observer equation admits a locally unique solution.

The reader is reminded that the optimal control problem and the corresponding value function discussed in \Cref{sec: RegOfValFun} and \Cref{sec: CharSecDer} are a shifted version of the one used for the derivation of the formal observer equation \eqref{eq: MorForm}, cf., \Cref{rem: ShiftOCP}. Hence we consider a shifted version of the observer equation that is obtained by subtracting the model dynamics from \eqref{eq: MorForm} and in its strong form reads
\begin{equation}\label{eq: MorShiftForm}
	\begin{aligned}
		\dot{\check{w}}(t) = \mathcal{A} \check{w}(t)
		&+ \begin{bmatrix}
			0 \\ - \check{w}_1^3(t) - 3 \Tilde{w}_1(t) \check{w}_1^2(t) - 3 \Tilde{w}_1^2(t) \check{w}_1(t)
		\end{bmatrix}\\
		&+ \alpha D_{\xi\xi}^2 \mathcal{V}(t,\check{w}(t),y)^{-1}
		(y(t) - C\check{w}(t)),\\
		\check{w}(0) &= 0. 
	\end{aligned}
\end{equation}
In order to apply a fixed point argument the equation needs to be rearranged by inserting a zero. The above equation for $\check{w}$ is equivalent to   
\begin{equation*}
	\begin{aligned}
		\dot{\check{w}}(t) &= \mathcal{A} \check{w}(t)
		+F(t) \check{w}(t)
		+ \begin{bmatrix}
			0 \\ - \check{w}_1^3(t) - 3 \Tilde{w}_1(t) \check{w}_1^2(t) 
		\end{bmatrix}
		+ \alpha D_{\xi\xi}^2 \mathcal{V}(t,\check{w}(t),y)^{-1} C^* y(t)\\
		&- \alpha \left( D_{\xi\xi}^2 \mathcal{V}(t,\check{w}(t),y)^{-1} 
		-D_{\xi\xi}^2 \mathcal{V}(t,0,0)^{-1}  \right)
		C^* C \check{w}(t),
		\\
		\check{w}(0) &= 0, 
	\end{aligned}
\end{equation*}
where
\begin{equation*} 
	F \colon [0,T] \rightarrow \mathcal{L}(\mathcal{E}),
	~~~
	F(t) = \begin{bmatrix} 0 & 0 \\ -3\Tilde{w}_1^2(t) \, \mathrm{Id} & 0 \end{bmatrix}
	- \alpha D_{\xi\xi}^2 \mathcal{V}(t,0,0)^{-1} C^* C.
\end{equation*}
The regularity of $\Tilde{w}$ and \Cref{cor: InvStrCont} ensure that $F$ is strongly continuous allowing a formulation of the equation in terms of
the evolution operator $U(t,s)$ associated with $\mathcal{A} + F$, cf., \cite[P.~II, Ch.~3.5]{BenEtAl07}.
We call $\check{w} \in \mathcal{C}_T$ a solution to the observer equation \eqref{eq: MorShiftForm} if for all $t \in [0,T]$ it holds
\begin{equation*}
	\begin{aligned}
		\check{w}(t)
		= \int_0^t 
		&U(t,s)
		\left(
		\begin{bmatrix} 0 \\ - \check{w}_1^3(s) - 3 \Tilde{w}_1(s) \check{w}_1^2(s) 
		\end{bmatrix}
		+ \alpha \, D_{\xi\xi}^2 \mathcal{V}(s,\check{w}(s),y)^{-1} C^* y(s)
		\right)
		\\
		- &U(t,s) \left(
		\alpha \left( D_{\xi\xi}^2 \mathcal{V}(s,\check{w}(s),y)^{-1} 
		-D_{\xi\xi}^2 \mathcal{V}(s,0,0)^{-1}  \right)
		C^* C \check{w}(s)
		\right)  \mathrm{d}s.
	\end{aligned}
\end{equation*}
Note that due to $D_{\xi\xi}^2 \mathcal{V}^{-1}$ local results obtained in previous sections ensure that the right hand side is well defined only for $\check{w}(s) < \delta_4$.
By application of a fixed point argument on a suitable subset of $\mathcal{C}_T$ we show that there exists such a $\check{w}$ satisfying the mild formulation.
The following lemma is crucial in showing the contraction property of the fixed point operator.

\begin{lemma}\label{lem: InvLip}
	Let $t \in [0,T]$ be fixed. There exists a constant $L>0$ independent of $t$ such that for all $(\xi_1,y_1)$, $(\xi_2,y_2) \in \mathcal{D}_{\delta_4}^t$ it holds
	\begin{equation*}
		\Vert D_{\xi\xi}^2 \mathcal{V}(t,\xi_2,y_2)^{-1} - D_{\xi\xi}^2 \mathcal{V}(t,\xi_1,y_1)^{-1} \Vert_{\mathcal{L}(\mathcal{E})} 
		\leq
		L \left( \Vert \xi_2 - \xi_1 \Vert_\mathcal{E}
		+ \Vert y_2 - y_1 \Vert_{\mathcal{Y}_t} \right).
	\end{equation*}
\end{lemma}
\begin{proof}
	Using the definition of the Fr\'echet derivative and the Neumann series it can easily be confirmed that the operator 
	\begin{equation*}
		\mathcal{F} \colon \mathcal{L}_\mathrm{coerc}(\mathcal{E})
		\rightarrow \mathcal{L}_\mathrm{coerc}(\mathcal{E}),
		~~~~~
		A \mapsto A^{-1}
	\end{equation*}
	is continuously differentiable with derivative $B \mapsto A^{-1} B A^{-1}$, where $\mathcal{L}_\mathrm{coerc}(\mathcal{E})$ denotes the subset of $\mathcal{L}(\mathcal{E})$ containing all elements associated with coercive bilinear forms. Due to \Cref{prop: D2Coerc} we have $D_{\xi\xi}^2 \mathcal{V}(t,\xi,y) \in \mathcal{L}_\mathrm{coerc}(\mathcal{E})$ for all $(\xi,y) \in \mathcal{D}_{\delta_4}^t$. Since $\mathcal{L}_\mathrm{coerc}(\mathcal{E})$ is open and convex, Taylor's theorem can be applied. We denote $\xi_\tau = \xi_1 + \tau (\xi_2 - \xi_1)$ and $y_\tau = y_1 + \tau (y_2 - y_1)$ and get
	\begin{equation*}
		\begin{aligned}
			&D_{\xi\xi}^2 \mathcal{V}(t,\xi_2,y_2)^{-1} 
			- D_{\xi\xi}^2 \mathcal{V}(t,\xi_1,y_1)^{-1}\\
			&= \int_0^1
			D_\xi \left( D_{\xi\xi}^2 \mathcal{V}(t,\xi_\tau,y_\tau)^{-1} \right) 
			[\xi_2 - \xi_1]
			+
			D_y \left[ D_{\xi\xi}^2 \mathcal{V}(t,\xi_\tau,y_\tau)^{-1} \right]
			[y_2 - y_1]
			\, \mathrm{d}\tau
			\\
			&= \int_0^1
			D_{\xi\xi}^2 \mathcal{V}(t,\xi_\tau,y_\tau)^{-1} \,
			D_{\xi^3}^3 \mathcal{V}(t,\xi_\tau,y_\tau) [\xi_2 - \xi_1] \,
			D_{\xi\xi}^2 \mathcal{V}(t,\xi_\tau,y_\tau)^{-1}
			\, \mathrm{d}\tau \\
			&+ \int_0^1 
			D_{\xi\xi}^2 \mathcal{V}(t,\xi_\tau,y_\tau)^{-1} \,
			D_{y \xi^2}^3 \mathcal{V}(t,\xi_\tau,y_\tau) [y_2 - y_1] \,
			D_{\xi\xi}^2 \mathcal{V}(t,\xi_\tau,y_\tau)^{-1}	
			\, \mathrm{d}\tau.
		\end{aligned}	
	\end{equation*}
	The assertion follows with \Cref{lem: ValFunDerEst} and \Cref{prop: D2Coerc}.
\end{proof}

We are now prepared to show that the observer equation admits a solution.
\begin{theorem}\label{thm: ObsEqWP}
	There exists $\delta_7 \in (0,\delta_4]$ and $\gamma>0$ such that for all $y \in \mathcal{Y}_T$ satisfying $\Vert y \Vert_{\mathcal{Y}_T} < \delta_7$ the observer equation \eqref{eq: MorShiftForm} admits exactly one solution $\check{w} \in \mathcal{C}_T$ with $\Vert \check{w} \Vert_{\mathcal{C}_T} < 2 \gamma \Vert y \Vert_{Y_T}$.
\end{theorem}
\begin{proof}
	For now let $\gamma >0$ be some constant and set $\delta_7 = \min ( \delta_4, \tfrac{\delta_4}{4\gamma}  )$. Let $y \in \mathcal{Y}_t$ be such that $\Vert y \Vert_{\mathcal{Y}_T} < \delta_7$ and define the set 
	\begin{equation*}
		\mathcal{M} \coloneqq
		\left\{
		w \in \mathcal{C}_T 
		\colon \Vert w \Vert_{\mathcal{C}_T} \leq 2 \gamma \Vert y \Vert_{\mathcal{Y}_T}
		\right\}
	\end{equation*}
	and the operator $\mathcal{Z} \colon \mathcal{M} \rightarrow \mathcal{C}_T$ as
	\begin{equation*}
		\begin{aligned}
			\mathcal{Z}(w)(t)
			&= \int_0^t 
			U(t,s)
			\left( \vphantom{\begin{bmatrix} 0 \\ - {w}_1^3(s) - 3 \Tilde{w}_1(s) {w}_1^2(s) 
			\end{bmatrix}} \right.
			\underbrace{
				\begin{bmatrix} 0 \\ - {w}_1^3(s) - 3 \Tilde{w}_1(s) {w}_1^2(s) 
				\end{bmatrix}
			}_{\eqqcolon \mathcal{G}_1[w](s)}
			\underbrace{
				\vphantom{\begin{bmatrix} 0 \\ - {w}_1^3(s) - 3 \Tilde{w}_1(s) {w}_1^2(s) 
				\end{bmatrix}}
				+ \alpha \, D_{\xi\xi}^2 \mathcal{V}(s,{w}(s),y)^{-1} C^* y(s)
			}_{\eqqcolon \mathcal{G}_2[w](s)}
			\left. \vphantom{\begin{bmatrix} 0 \\ - {w}_1^3(s) - 3 \Tilde{w}_1(s) {w}_1^2(s) 
			\end{bmatrix}} \right)
			\\
			&~~~~~~ - U(t,s) 
			\underbrace{
				\left(
				\alpha \left( D_{\xi\xi}^2 \mathcal{V}(s,{w}(s),y)^{-1} 
				-D_{\xi\xi}^2 \mathcal{V}(s,0,0)^{-1}  \right)
				C^* C {w}(s)
				\right) 
			}_{\eqqcolon \mathcal{G}_3[w](s)}	
			\mathrm{d}s.
		\end{aligned}
	\end{equation*}
	Note that any $w \in \mathcal{M}$ satisfies $\Vert w \Vert_{\mathcal{C}_T} \leq \tfrac{1}{2} \delta_4 < \delta_4 $ ensuring that for all $t \in [0,T]$ the expressions $D_{\xi\xi}^2 \mathcal{V}(t,w(t),y)^{-1}$ and $D_{\xi\xi}^2 \mathcal{V}(t,0,0)^{-1}$ are well-defined and bounded uniformly in $t$, cf., \Cref{prop: D2Coerc}. It follows $\mathcal{G}_1 + \mathcal{G}_2 + \mathcal{G}_3 \in L^2(0,T;\mathcal{E})$ and $\mathcal{Z}$ is well-defined.
	
	By showing that $\mathcal{Z}$ is a contraction on the closed subset $\mathcal{M}$ of the Banach space $\mathcal{C}_T$ we ensure existence of a unique fixed point $\check{w} \in \mathcal{M}$ which then solves the equation and satisfies the estimate. We begin by showing that $\mathcal{Z}$ maps to $\mathcal{M}$. To that end we first note that according to \cite[P.II, Prop.~3.6]{BenEtAl07} there exists $c_{U} > 0$ such that for all $t \in [0,T]$, $s \in [0,t]$ it holds $\Vert U(t,s) \Vert_{\mathcal{L}(\mathcal{E})} \leq c_U$. We stress that the constant depends on $M_0$ from \Cref{prop: modelCoerc} but is independent of $\delta_4$ and $\delta_7$.
	Now let $w \in \mathcal{M}$ and for $t \in [0,T]$ consider
	\begin{equation*}
		\begin{aligned}
			\Vert \mathcal{Z}(w)(t) \Vert_\mathcal{E}
			\leq 
			c_U \int_0^t 
			\Vert \mathcal{G}_1[w](s) \Vert_\mathcal{E}
			+ \Vert \mathcal{G}_2[w](s) \Vert_\mathcal{E}
			+ \Vert \mathcal{G}_3[w](s) \Vert_\mathcal{E}
			\, \mathrm{d}s.
		\end{aligned}
	\end{equation*}
	Considering the terms individually we have 
	\begin{equation*}
		\Vert \mathcal{G}_1[w](s) \Vert_\mathcal{E}
		\leq \Vert w_1^3(s) \Vert_{L^2(\Omega)} 
		+ 3 \Vert \Tilde{w}_1(s) w_1^2(s) \Vert_{L^2(\Omega)}
		\leq c_\mathrm{em}^3 \left(
		\Vert w \Vert_{\mathcal{C}_T}^3 
		+ 3 c_{\Tilde{w}} \Vert w \Vert_{\mathcal{C}_T}^2
		\right)
	\end{equation*}
	followed by 
	\begin{equation*}
		\Vert \mathcal{G}_2[w](s) \Vert_\mathcal{E}
		\leq \alpha c_C M^{-1} \,
		\Vert y(s) \Vert_{Y}
	\end{equation*}
	and finally
	\begin{equation*}
		\Vert \mathcal{G}_3[w](s) \Vert_\mathcal{E}
		\leq 
		\alpha c_C^2 L \left(
		\Vert w(s) \Vert_\mathcal{E} + \Vert y (s) \Vert_{Y}
		\right)
		\Vert w(s) \Vert_\mathcal{E},
	\end{equation*}
	where we utilized \Cref{prop: D2Coerc} and \Cref{lem: InvLip}. Combining the estimates and applying Young's inequality we obtain
	\begin{equation*}
		\begin{aligned}
			\Vert \mathcal{Z}(w)(t) \Vert_\mathcal{E} 
			&\leq 
			c_U T c_\mathrm{em}^3 \left(
			\Vert w \Vert_{\mathcal{C}_T}^3 
			+ 3 c_{\Tilde{w}} \Vert w \Vert_{\mathcal{C}_T}^2
			\right)
			+
			c_U \alpha c_C \sqrt{T} M^{-1} 
			\Vert y \Vert_{\mathcal{Y}_T}\\
			&+ c_U c_C^2 \alpha L T
			\Vert w \Vert_{\mathcal{C}_T}^2
			+ \tfrac{1}{2} c_U c_C^2 \alpha L
			\left( \Vert y \Vert_{\mathcal{Y}_T}^2 + T \Vert w \Vert_{\mathcal{C}_T}^2 \right).
		\end{aligned}
	\end{equation*}
	Since none of the appearing constants depend on $\delta_4$ or $\delta_7$, we can specify $\gamma > 0$ to be such that
	\begin{equation*}
	\begin{aligned}
		\Vert \mathcal{Z}(w)(t) \Vert_\mathcal{E} 
		&\leq \gamma \left(
		\Vert y \Vert_{\mathcal{Y}_T}
		+ \Vert y \Vert_{\mathcal{Y}_T}^2
		+ \Vert w \Vert_{\mathcal{C}_T}^2
		+ \Vert w \Vert_{\mathcal{C}_T}^3
		\right)\\
		&\leq \gamma \left(
		\Vert y \Vert_{\mathcal{Y}_T}
		+ \Vert y \Vert_{\mathcal{Y}_T}^2
		+ 4 \gamma^2 \Vert y \Vert_{\mathcal{Y}_T}^2
		+ 8 \gamma^3 \Vert y \Vert_{\mathcal{Y}_T}^3
		\right),
	\end{aligned}
	\end{equation*}
	where the second estimate is justified by $w \in \mathcal{M}$. Finally the assumption on $y$ and a possible decrease of $\delta_7$ such that $ \delta_7 (1 + 4 \gamma^2 + 8 \gamma^3 \delta_7) \leq 1$ yield $\mathcal{Z}(w) \in \mathcal{M}$.
	
	To show the contraction property let $w,z \in \mathcal{M}$ and for $t \in [0,T]$ consider
	\begin{equation*}
		\begin{aligned}
			\Vert \mathcal{Z}(w)(t) - \mathcal{Z}(z)(t) \Vert_\mathcal{E}
			&\leq
			c_U \int_0^t
			\Vert \mathcal{G}_1[w](s) - \mathcal{G}_1[z](s) \Vert_{\mathcal{E}}\\
			&+
			\Vert \mathcal{G}_2[w](s) - \mathcal{G}_2[z](s) \Vert_{\mathcal{E}}
			+
			\Vert \mathcal{G}_3[w](s) - \mathcal{G}_3[z](s) \Vert_{\mathcal{E}}
			\, \mathrm{d}s.
		\end{aligned}
	\end{equation*}
	Using the identities \eqref{eq: cubId} we find
	\begin{equation*}
	\begin{aligned}
		&\Vert \mathcal{G}_1[w](s) - \mathcal{G}_1[z](s) \Vert_{\mathcal{E}}\\
		&\leq 3 c_\mathrm{em}^3
		\left(  
		\tfrac{1}{3}\Vert w - z \Vert_{\mathcal{C}_T}^2
		+  \Vert z \Vert_{\mathcal{C}_T} \Vert w- z \Vert_{\mathcal{C}_T}
		+  \Vert z \Vert_{\mathcal{C}_T}^2
		+  c_{\Tilde{w}} \Vert w + z \Vert_{\mathcal{C}_T} 
		\right) \Vert w - z \Vert_{\mathcal{C}_T}.
	\end{aligned}
	\end{equation*}
	Using \Cref{lem: InvLip} we obtain
	\begin{equation*}
		\Vert \mathcal{G}_2[w](s) - \mathcal{G}_2[z](s) \Vert_{\mathcal{E}}
		\leq
		\alpha c_C L
		\Vert y(s) \Vert_Y
		\Vert w(s) - z(s) \Vert_\mathcal{E} 
	\end{equation*}
	and by inserting a zero
	\begin{equation*}
		\begin{aligned}
			\Vert \mathcal{G}_3[w](s) - \mathcal{G}_3[z](s) \Vert_{\mathcal{E}}
			&\leq
			\alpha \left\Vert \left(
			D_{\xi\xi}^2 \mathcal{V}(s,{w}(s),y)^{-1}
			-
			D_{\xi\xi}^2 \mathcal{V}(s,0,0)^{-1}
			\right) \, C^* C w(s) \right. \\
			&- \left. \left( 
			D_{\xi\xi}^2 \mathcal{V}(s,w(s),y)^{-1} 
			-
			D_{\xi\xi}^2 \mathcal{V}(s,0,0)^{-1}
			\right) \, C^* C z(s)
			\right\Vert_\mathcal{E}\\
			&+
			\alpha \left\Vert \left(
			D_{\xi\xi}^2 \mathcal{V}(s,{w}(s),y)^{-1}
			-
			D_{\xi\xi}^2 \mathcal{V}(s,0,0)^{-1}
			\right)  \, C^* C z(s) \right.\\
			&- \left. \left(
			D_{\xi\xi}^2 \mathcal{V}(s,z(s),y)^{-1} 
			-
			D_{\xi\xi}^2 \mathcal{V}(s,0,0)^{-1}
			\right) \, C^* C z(s)
			\right\Vert_\mathcal{E}\\
			&\leq 
			\alpha c_C^2 L
			\left( \Vert w(s) \Vert_\mathcal{E} + \Vert y(s) \Vert_Y + \Vert z(s) \Vert_\mathcal{E}  \right)
			\Vert w(s) - z(s) \Vert_\mathcal{E}.\\
		\end{aligned}
	\end{equation*}
	These three estimates together with $\Vert y \Vert_{\mathcal{Y}_T}< \delta_7$ and $w,z \in \mathcal{M}$ and a possible decrease of $\delta_7$ ensure that $\mathcal{Z}$ is a contraction.
\end{proof}
%

%
\section{Conclusion}\label{sec: Concl}
%
In this work we took a first step towards establishing well-posedness of the Mortensen observer applied to infinite-dimensional systems. For the nonlinear wave equation at hand it was shown that the state estimation via the energy minimizer is well-defined and further that the associated observer equation admits a locally unique solution. It still remains to show that the energy minimizer coincides with this trajectory, a task that could possibly be tackled analyzing the associated Hamilton-Jacobi-Bellman equation and the regularity of its solution. Another strain of future research is the extension to more general systems. Since the arguments of the present work rely on the reversibility of the dynamics, hyperbolic systems are the natural first candidate. However, one might be able to also treat irreversible dynamics and parabolic behavior by limiting the search for a minimizer of the energy to a subset of states reachable by the dynamics, hence allowing for a backward solution. 

%
%
%
\appendix
%
\section{Linear quadratic optimal control problem}\label{sec: App}~\\
%
In the appendix we discuss a general linear quadratic optimal control problem. It reads
\begin{equation}\label{eq: GLQR}\tag{GLQR}
	\min\limits_{(w,v) \in \mathcal{C}_t \times \mathcal{L}_t} \Tilde{J}(w,v)
	~~\text{subject to}~~ w = D_v \mathcal{S}^t(\bar{v},\xi) [v] + f,
\end{equation}
where
\begin{equation}\label{eq: GLQRcost}
	\begin{aligned}
		\Tilde{J}(w,v) 
		&\coloneqq \tfrac{1}{2} \Vert w(0) \Vert_{\mathcal{E}}^2
		+ \tfrac{1}{2} \Vert v \Vert_{\mathcal{L}_t}^2
		+ \tfrac{\alpha}{2} \Vert \gamma - C w \Vert_{\mathcal{Y}_t}^2
		+ \tfrac{1}{2} \left( \bar{w}(0) - w_0, D_{vv}^2 \mathcal{S}^t(\bar{v},\xi) [v,v] (0) \right)_\mathcal{E}\\
		&- \tfrac{\alpha}{2} \left\langle \Psi[\bar{w} - \Tilde{w},y],
		D_{vv}^2 \mathcal{S}^t(\bar{v},\xi) [v,v] \right\rangle_{\mathcal{C}_t^*}
		+ \left( L_1,v \right)_{\mathcal{L}_t}
		+ \left( L_2,v \right)_{\mathcal{L}_t},\\
	\end{aligned}
\end{equation}
and $t \in (0,T]$, $\xi \in \mathcal{E}$, $\bar{v} \in \mathcal{L}_t$, $y \in \mathcal{Y}_t$, $f \in \mathcal{C}_t$, $\gamma \in \mathcal{Y}_t$, $L_1 \in \mathcal{L}_t$, $L_2 \in \mathcal{L}_t$ are given. We assume $\max \left( \Vert \xi \Vert_\mathcal{E}, \Vert \bar{v} \Vert_{\mathcal{L}_t}, \Vert y \Vert_{\mathcal{Y}_t} \right) < \delta^\prime$, with a constant $\delta^\prime \in (0,\delta_2]$ which properties will be specified later. Further $C \in \mathcal{L}(\mathcal{E};Y)$, $\mathcal{S}^t \colon \mathcal{L}_t \times \mathcal{E} \rightarrow \mathcal{C}_t$, and $\Psi[w,y] \in \mathcal{C}_t^*$ are as above. We denote $\bar{w} = \mathcal{S}^t(\bar{v},\xi)$ and by $\Tilde{w}$ the nominal trajectory.
We show that under suitable assumptions on $\delta^\prime$ \eqref{eq: GLQR} admits exactly one solution and establish the optimality system. 
\begin{proposition}\label{prop: GLQR}
	There exists $\delta^\prime \in (0,\delta_2]$ independent of $t$ such that for any $(\xi,\bar{v},y)$ satisfying $\max \left( \Vert \xi \Vert_\mathcal{E}, \Vert \bar{v} \Vert_{\mathcal{L}_t}, \Vert y \Vert_{\mathcal{Y}_t} \right) < \delta^\prime$ and every $f \in \mathcal{C}_t$, $\gamma \in \mathcal{Y}_t$, $L_1 \in \mathcal{L}_t$, $L_2 \in \mathcal{L}_t$ the system
	\begin{equation}\label{eq: OptSysGLQR}
		\begin{aligned}
			w &= D_v \mathcal{S}^t(\bar{v},\xi) [v] + f\\
			p &= D_v \mathcal{S}^t(\bar{v},\xi)^* (\delta_0^*(w(0)) - \alpha \Psi[w,\gamma])\\
			&+ \left( D_{vv}^2 \mathcal{S}^t(\bar{v},\xi) [v] \right)^*
			\left( \delta_0^*(\bar{w}(0) - w_0) - \alpha \Psi[\bar{w} - \Tilde{w},y] \right)
			+ L_1\\
			v &= -p - L_2
		\end{aligned}
	\end{equation}
	admits a unique solution $(w^*,v^*,p^*) \in \mathcal{C}_t \times \mathcal{L}_t \times \mathcal{L}_t$ and $(w^*,v^*)$ is the unique minimizer of \eqref{eq: GLQR}. Further there exists $\bar{c}>0$ independent of $t$ such that 
	\begin{equation*}
		\max \left( \Vert w^* \Vert_{\mathcal{C}_t}, \Vert v^* \Vert_{\mathcal{L}_t}, \Vert p^* \Vert_{\mathcal{L}_t} \right)
		\leq 
		\bar{c} \left( \Vert f \Vert_{\mathcal{C}_t} + \Vert \gamma \Vert_{\mathcal{Y}_t} + \Vert L_1 \Vert_{\mathcal{L}_t} + \Vert L_2 \Vert_{\mathcal{L}_t} \right).
	\end{equation*}
\end{proposition}
\begin{proof}
	Instead of \eqref{eq: GLQR} we analyze the reduced problem given by
	\begin{equation*}
		\min\limits_{v \in \mathcal{L}_t} \Tilde{J}^r(v) \coloneqq \Tilde{J}(w_v,v),
	\end{equation*}
	where for $v \in \mathcal{L}_t$ we denote $w_v = D_v \mathcal{S}^t(\bar{v},\xi) [v] + f$. In the following we show that $\Tilde{J}^r$ is strictly convex and weakly coercive. Then \cite[Sec.~2.9, Prop.~1]{Zei95AMS109} yields that \eqref{eq: GLQR} admits exactly one solution and $v \in \mathcal{L}_t$ is the solution if and only if 
	\begin{equation*}
		D \Tilde{J}^r(v) = 0.
	\end{equation*}
	To see that $\Tilde{J}^r$ is indeed strictly convex and weakly coercive we first note that $w_v$ depends on $v$ in an affine-linear fashion. Considering the definition of $\Tilde{J}$ we conclude that the linear and non-negative quadratic terms are convex. Due to \Cref{prop: SolStateEq}, \Cref{lem: EstS}, and \eqref{eq: estPsi1} we can choose time independent constants $\delta^\prime \in (0,\delta_2] $ and $\epsilon > 0$  such that for all $v \in \mathcal{L}_t$ it holds
	\begin{equation*}
		\begin{aligned}
			&\Vert v \Vert_{\mathcal{L}_t}^2
			+ \left( \bar{w}(0) - w_0, D_{vv}^2 \mathcal{S}^t(\bar{v},\xi) [v,v] (0) \right)_\mathcal{E}
			- \alpha \left\langle \Psi[\bar{w} - \Tilde{w},y],
			D_{vv}^2 \mathcal{S}^t(\bar{v},\xi) [v,v] \right\rangle_{\mathcal{C}_t^*}\\
			&\geq \epsilon \Vert v \Vert_{\mathcal{L}_t}^2.
		\end{aligned}
	\end{equation*}
	Hence the asserted weak coercivity is shown. The strict convexity is ensured analogously. One easily calculates that $D \Tilde{J}^r (v) =0$ if and only if $w_v$ and $v$ together with the appropriate $p$ solve \eqref{eq: OptSysGLQR}.
	
	It remains to show the estimates. With Young's inequality we obtain
	\begin{equation*}
		\begin{aligned}
			\frac{\epsilon}{2} \Vert v^* \Vert^2 
			&\leq  \Tilde{J}^r(v^*)
			- (L_1 + L_2,v^*)_{\mathcal{L}_t} 
			\leq \Tilde{J}^r(0)
			+ \Vert L_1 \Vert_{\mathcal{L}_t} \Vert v^* \Vert_{\mathcal{L}_t}
			+ \Vert L_2 \Vert_{\mathcal{L}_t} \Vert v^* \Vert_{\mathcal{L}_t}\\
			&\leq \frac{1}{2} \Vert f(0) \Vert_{\mathcal{E}}^2
			+ \frac{\alpha}{2} \Vert \gamma - Cf \Vert_{\mathcal{Y}_t}^2
			+ \frac{2}{\epsilon} \Vert L_1 \Vert_{\mathcal{L}_t}^2
			+ \frac{\epsilon}{4} \Vert v^* \Vert_{\mathcal{L}_t}^2
			+ \frac{2}{\epsilon} \Vert L_2 \Vert_{\mathcal{L}_t}^2,
		\end{aligned}
	\end{equation*}
	which implies the estimate for $v^*$. The estimates for $w^*$ and $p^*$ follow by standard calculations utilizing \Cref{prop: SolStateEq}, \Cref{lem: EstS}, and \eqref{eq: estPsi1}.
\end{proof}
%
\section*{Acknowledgments}
%
We thank T. Breiten (TU Berlin) and K. Kunisch (KFU Graz, RICAM Linz) for valuable feedback on earlier versions of the manuscript. Funded by the Deutsche Forschungsgemeinschaft (DFG, German Research Foundation) - 504768428.

\bibliographystyle{siam}
\bibliography{references} 

\end{document}